\documentclass[11pt]{amsproc}

\usepackage{mathtext}
\usepackage[cp1251]{inputenc}

\usepackage{bm}

\usepackage[dvips]{graphicx}
\usepackage{amsmath}
\usepackage{amssymb}
\usepackage{amsxtra}

\usepackage{epsfig}
\usepackage{epic}
\usepackage{eepic}
\usepackage{graphics}
\usepackage{graphicx}
\usepackage{subfigure}

\usepackage{caption}
\def\N{{{\Bbb N}}}
\def\Z{{{\Bbb Z}}}
\def\T{{{\Bbb T}}}
\def\R{{\Bbb R}}

\def\l{{\lambda }}
\def\a{{\alpha }}
\def\D{{\Delta }}

\def\a{{\alpha}}
\def\b{{\beta}}
\def\d{{\delta}}

\def\vp{{\varphi}}
\def\t{{\theta }}
\def\g{{\gamma }}
\def\w{{\omega }}

\def\){\right)}
\def\({\left(}

\numberwithin{equation}{section}
\newtheorem{corollary}{Corollary}[section]
\newtheorem{lemma}{Lemma}[section]
\newtheorem{theorem}{Theorem}[section]
\newtheorem{proposition}{Proposition}[section]
\newtheorem{remark}{Remark}[section]
\newtheorem{example}{Example}[section]

\newcommand{\ww}{\widetilde}

\par

\begin{document}

\title[]{Sharp $L_p$-error estimates for sampling operators}

\author[Yurii
Kolomoitsev]{Yurii
Kolomoitsev$^{\text{a, 1, 2}}$}
\address{Institute for Numerical and Applied Mathematics, G\"ottingen University, Lotzestr. 16-18, 37083 G\"ottingen, Germany}
\email{kolomoitsev@math.uni-goettingen.de}

\author[Tetiana
Lomako]{Tetiana
Lomako$^{\text{a, 1}}$}
\address{Institute for Numerical and Applied Mathematics, G\"ottingen University, Lotzestr. 16-18, 37083 G\"ottingen, Germany}
\email{t.lomako@math.uni-goettingen.de}

\thanks{$^\text{a}$Institute for Numerical and Applied Mathematics, G\"ottingen University, Lotzestr. 16-18, 37083 G\"ottingen, Germany}

\thanks{$^1$Supported by the German Research Foundation, project KO 5804/1-2}
\thanks{$^2$Support by the German Research Foundation in the framework of the RTG 2088}

\thanks{$^*$Corresponding author}

\thanks{E-mail address: kolomoitsev@math.uni-goettingen.de}

\date{\today}

\begin{abstract}
We study approximation properties of linear sampling operators in the spaces $L_p$ for $1\le p<\infty$.
By means of the Steklov averages, we introduce a new measure of smoothness that simultaneously contains information on the smoothness of a function in~$L_p$ and discrete information on the behaviour of a function at sampling points. The new measure of smoothness enables us to improve and extend several classical results of approximation theory to the case of linear sampling operators. In particular, we  obtain matching direct and inverse approximation inequalities for sampling operators in~$L_p$, find the exact order of decay of the corresponding $L_p$-errors for particular classes of functions, and introduce a special $K$-functional and its realization suitable for studying smoothness properties of sampling operators.
\end{abstract}

\maketitle

\section{Introduction}

Let $\T=[0,1)$ and $1\le p<\infty$. By $L_p(\T)$ we denote the space of all 1-periodic measurable finite valued functions $f$ on $\T$ such that
$$
\|f\|_p=\(\int_\T |f(x)|^pdx\)^{1/p}<\infty.
$$
Note that we do not identify functions coinciding almost everywhere. Thus, any function in $L_p(\T)$ is given by its values at each point of $\T$.
By $B(\T)$ we denote the set of all measurable bounded 1-periodic functions with the norm
$\|f\|_\infty=\max_{x\in\T}|f(x)|.$ By $C(\T)$ we denote the set of all 1-periodic continuous functions.  As usual, $f\in W_p^r(\T)$ for some $r\in \N$ if $f^{(r-1)}$ is absolutely continuous and $f^{(r)}\in L_p(\T)$.
For $f\in L_p(\T)$, $1\le p<\infty$, and a set of points  $X_n=(x_{k,n})_{k=1}^{n}\subset\T$, the discrete semi-norm $\|f\|_{\ell_p(X_n)}$ is defined by
$$
\|f\|_{\ell_p(X_n)}=\(\frac1n\sum_{k=1}^n|f(x_{k,n})|^p\)^{1/p}.
$$

Let $\mathcal{T}_n$ denote the set of all trigonometric polynomials of degree at most $n$. The error of best approximation of $f\in L_p(\T)$ and the error of best one-sided approximation are given by
$$
E_n(f)_p=\inf\{\Vert f-T_n\Vert_p\,:\, T_n\in \mathcal{T}_n\}
$$
and
$$
\widetilde{E}_n(f)_p=\inf\{\Vert Q-q\Vert_p\,:\,  Q,q \in \mathcal{T}_n,\quad q(x)\le f(x)\le Q(x)\},
$$
respectively.

The integral  (classical) modulus of smoothness of a function $f\in L_p(\T)$, $1\le p\le\infty$,
of order $r\in \N$ and step $\d>0$ is defined by
\begin{equation*}
    \w_r(f,\d)_p=\sup_{0<h<\d} \Vert \D_h^r f\Vert_p,
\end{equation*}
where
$$
\D_h^r f(x)=\sum_{\nu=0}^r\binom{r}{\nu}(-1)^{\nu} f(x+(r-\nu)h),
$$
$\binom{r}{\nu}=\frac{r (r-1)\dots (r-\nu+1)}{\nu!},\quad \binom{r}{0}=1$. In the case $r=1$, we will write $\w(f,\d)_p=\w_1(f,\d)_p$.
The averaged modulus of smoothness ($\tau$-modulus) of a function $f\in L_p(\T)$, $0< p\le\infty$,
of order $r\in \N$ and step $\d>0$ is defined by
$$
\tau_r(f,\d)_p=\|\w_r(f,\cdot,\d)\|_p=\(\int_\T (\w_r(f,x,\d))^pdx\)^{1/p},
$$
where
$$
\w_r(f,x,\d)=\sup\left\{|\D_\d^r f(t)|\,:\,t,t+rh\in [x-r\d/2,x+r\d/2]\right\}
$$
is the local modulus of smoothness of $f$.  In the case $r=1$, we will write $\tau(f,\d)_p=\tau_1(f,\d)_p$.

Consider the trigonometric Lagrange interpolation polynomials:
\begin{equation}\label{Lag}
  L_n(f)(x)=\frac1{2n+1}\sum_{k=0}^{2n} f(t_k)D_n(x-t_k),\quad t_k=\frac{k}{2n+1},
\end{equation}
where
$$
D_n(x)=\sum_{\ell=-n}^n e^{2\pi i\ell x}
$$
is the classical Dirichlet kernel. It is well known that
\begin{equation}\label{i1}
  \|f-L_n(f)\|_\infty\le C\log(n+1)E_n(f)_\infty,\quad  f\in C(\T),\quad n\in \N,
\end{equation}
where $C$ is some absolute constant. At the same time, there exists a continuous function $f_0$ such that
$$
\underset{n\to \infty}{\lim\sup}\,|L_n(f_0)(x)|=\infty
$$
for almost all $x\in \T$ (see, e.g.,~\cite[Theorem~8.14, Ch.~X]{Z}). See also~\cite{V82} for a generalization of this result to arbitrary families of knots of interpolation.

In the case of approximation in the spaces $L_p(\T)$ with $1<p<\infty$, we have a different picture. It is known (see, e.g.,~\cite[Theorems 7.14 and 7.1, Ch.~X]{Z}, see also~\cite{L40}) that  if a function $f$ is Riemann integrable and $1<p<\infty$, then
\begin{equation*}
  \|f-L_n(f)\|_p\to 0\quad\text{as}\quad n\to \infty.
\end{equation*}
Moreover, it follows from~\cite[Theorems 7.14, Ch.~X]{Z} that, for any $f\in C(\T)$ and $1<p<\infty$,
\begin{equation}\label{i2}
  \|f-L_n(f)\|_p\le C(p)E_n(f)_\infty,\quad n\in \N,
\end{equation}
cf.~\eqref{i1}. There are many papers dedicated to improvements and generalizations of estimate~\eqref{i2}, see, e.g., \cite{CZ93_2}, \cite{CZ93}, \cite{CZ99}, \cite{H83}, \cite{H}, \cite{KLP}, \cite{O86}, \cite{Pr84}, \cite{PrXu}.
The most general result is given in terms of the error of best one-sided approximation and averaged moduli of smoothness. In particular, it was proved in~\cite{H}, see also~\cite{PrXu}, that if $f\in B(\T)$, $1<p<\infty$, and $s\in \N$, then
\begin{equation}\label{i5}
  \|f-L_n(f)\|_p\le C(p)\ww E_n(f)_p\le C(p,s)\tau_s(f,1/n)_p,\quad n\in \N.
\end{equation}
For more general sampling oparators, analogues of~\eqref{i5} were established e.g. in~\cite[Theorem~2.9]{SP}, see also~\cite{CZ93_2}, \cite{CZ93}, \cite{CZ99}, \cite{KP21}.

It follows from~\eqref{i5} and properties of the moduli of smoothness and the errors of best approximation (see, e.g., Lemmas~\ref{leJB} and~\ref{leJBt} and $(e_\tau)$, $(f_\tau)$ below)  that,
for all $f\in C(\T)$, $1<p<\infty$, and $\a\in (1/p,s)$,
the following properties are equivalent:
    \begin{enumerate}
    \item[$(i_1)$]  $\|f-L_n(f)\|_p=\mathcal{O}(n^{-\a})$,
    \item[$(i_2)$]  $\ww E_n(f)_p=\mathcal{O}(n^{-\a})$,
    \item[$(i_3)$] $\tau_s(f,\d)_p=\mathcal{O}(\d^{-\a})$,
    \item[$(i_4)$] $\w_s(f,\d)_p=\mathcal{O}(\d^{-\a})$,
    \item[$(i_5)$]  $E_n(f)_p=\mathcal{O}(n^{-\a})$.
  \end{enumerate}
Some of these properties are also equivalent for $\a>0$ and wider classes of functions $f$. In particular, if $f\in B(\T)$ and $\a\in (0,s)$, then we have the following implications:
$$
(i_2) \Leftrightarrow (i_3) \Rightarrow (i_1) \Rightarrow (i_5) \Leftrightarrow (i_4).
$$

Although the inequalities in~\eqref{i5} provide the implications $(i_k) \Rightarrow (i_1)$, $k=2,3$, for all $\a>0$, 
these inequalities are not sharp in general and cannot be applied for some classes of functions. 
First, for each $k\ge 2$, equivalences $(i_1) \Leftrightarrow (i_k)$ are not valid if $\a\le 1/p$, see, e.g., examples in Section~5 below.  Second, the error of best one-sided approximation and the averaged moduli of smoothness make sense only for bounded functions on $\T$ while the interpolation polynomials $L_n(f)$ are well defined and may converge for  finite-valued (non necessarily bounded) functions, see Example~\ref{ex2}. Third, inequality~\eqref{i5} is not sharp even for simplest discontinuous functions. Indeed, if $f_0$ is a 1-periodic function such that $f_0(x)=1$ at $x=1/2$ and $f_0(x)=0$ for $x\in \T\setminus \{1/2\}$, then $0<\ww E_n(f_0)_p\asymp \tau_s(f_0,1/n)_p\asymp n^{-1/p}$ while $\|f_0-L_n(f_0)\|_p=0$ for each $n\in\N$. Of course, after the correction of $f$ on a set of measure zero, we do not have such a problem. Nevertheless, it is not difficult to find a continuous function $f$ such that
$$
\underset{n\to \infty}{\lim\sup}\frac{\ww E_n(f)_p}{\|f-L_n(f)\|_p}=\infty
$$
(see, e.g., Example~\ref{pr5} below).

It occurs the following natural question: what is the most suitable measure of smoothness for estimating the error of approximation  by the Lagrange interpolation polynomials in $L_p(\T)$ with $1\le p<\infty$? In particular, what can one use in~\eqref{i5} instead of~$\tau_s(f,1/n)_p$ and~$\ww E_n(f)_p$ to improve the above mentioned drawbacks? Note that in general $\ww E_n(f)_p$ cannot be replaced  by either $E_n(f)_p$ or $\w_s(f,1/n)_p$. For example, in~\cite{WS03} it was proved that there exists a 1-periodic infinitely differentiable function $f$ such that
\begin{equation}\label{i6}
  \underset{n\to \infty}{\lim\sup}\,\frac{\|f-L_n(f)\|_p}{\w(f,1/n)}=\infty,\quad 1<p<\infty.
\end{equation}
See also~\cite{BXZ92} for a more general result with the error of best approximation $E_n(f)_p$.
In spite of~\eqref{i6}, we show that it is possible to modify the integral modulus of smoothness  such that its modification is able to provide sharp estimates of $\|f-L_n(f)\|_p$. To introduce such a modification, we employ the well-known Steklov averages
$$
f_\d(x)=\frac1{\d}\int_{-\d/2}^{\d/2}f(x+t)dt,\quad \d>0.
$$
Recall (see~\cite{T80}, see also~\cite[8.2.5 and 8.2.8]{TB} and~\cite{DI93}) that, for any $f\in L_p(\T)$,  $1\le p<\infty$, we have
\begin{equation*}
  \|f_\d-f\|_p\asymp \w_2(f,\d)_p,\quad \d>0,
\end{equation*}
where $\asymp$ is a two-sided inequality with positive constants depending only on $p$. Thus, the quantity $\|f_\d-f\|_p$ has almost the same behaviour as the integral modulus of smoothness $\w_2(f,\d)_p$ but, in some particular problems, $\|f_\d-f\|_p$ is much more convenient  to handle than $\w_2(f,\d)_p$. Especially this concerns the problems, where there is a need to use methods of Fourier multipliers. For example, quantities similar to  $\|f_\d-f\|_p$ are very useful in establishing sharp error estimates of approximation by some classical and special approximation processes, see, e.g.,~\cite{K12}, \cite{KT12}, \cite{T80}, \cite{T13}, see also~\cite[Ch.~8]{TB}.

In this work, we modify $\|f_\d-f\|_p$ by replacing the integral norm $\|\cdot\|_p$ by its discrete analogue $\|\cdot\|_{\ell_p(X_n)}$, where $X_n=(x_{k,n})_{k=1}^n$ is a set of distinct points in $\T$, and consider the special "modulus of smoothness" 
\begin{equation*}
  \Omega(f,X_n)_p=\|f_\d-f\|_{\ell_p(X_n)}+\w_2(f,\d)_p.
\end{equation*}
Unlike to the conventional measures of smoothness, the modulus $\Omega(f,X_n)_p$ simultaneously contains information on the smoothness of $f$ in $L_p$ and discrete information on the behaviour of the function $f$ at the points $X_n$. This feature makes $\Omega(f,X_n)_p$ a right quantity for estimating the $L_p$-error of approximation of $f$ by the Lagrange interpolation polynomials and other general sampling processes. In particular, using $\Omega(f,X_n)_p$, we  obtain matching direct and inverse inequalities in the case of approximation of $f$ by the interpolation polynomials $L_n(f)$. Thus, in Corollary~\ref{cor2} below, we show that, for all $1<p<\infty$ and $\a \in (0,2)$, the following equivalence holds:
$$
\|f-L_n(f)\|_p=\mathcal{O}(n^{-\a})\quad \Longleftrightarrow\quad \Omega(f,X_n)_p=\mathcal{O}(n^{-\a}).
$$
As far as we know, results of this type for interpolation and sampling operators are new in the case $0<\a\le 1/p$.  In terms of $\Omega(f,X_n)_p$, we also find the exact order of decay of the corresponding $L_p$-errors for particular classes of functions
 as well as introduce a $K$-functional and its realization, which are equivalent to $\Omega(f,X_n)_p$.
In this paper, we do not restrict ourselves to the case of the trigonometric Lagrange polynomials and the modulus of smoothness of second order 
 rather we consider general linear sampling operators satisfying certain natural properties and measure the $L_p$-error of approximation by means of moduli of smoothness and Steklov averages of arbitrary integer order.

The paper is organized as follows. In Section~2 we define the main objects of the paper: general sampling operators and generalized Steklov averages. We also give  several important examples of sampling operators and discuss basic properties of Steklov averages.  Section~3 is devoted to auxiliary results.
In Section~4 we prove the main results: in Subsection~4.1, we establish the direct and inverse inequalities as well as find conditions ensuring the exact order of approximation by sampling operators in $L_p$; in Subsection~4.2  we consider $K$-functionals and their realizations as well as discuss smoothness properties of sampling operators; in Subsection~4.3 we prove several additional results. In Section~5 we give some important examples, which in particular demonstrate applicability of our results to different classes of functions as well as show some advantages of the new measures of smoothness over the averaged moduli of smoothness and the errors of best one-sided approximation for studying $L_p$-approximation by sampling operators.


\section{Main definitions and preliminary results}

Consider the general linear sampling operator $G_n$ defined  by
\begin{equation*}
  G_n(f)(x)=\sum_{k=1}^n f(x_{k,n})\vp_{k,n}(x),
\end{equation*}
where $\vp_{k,n}$ are appropriate 1-periodic functions (e.g., trigonometric polynomials, splines, etc.) and the set of points $X_n=(x_{k,n})_{k=1}^{n}$ satisfies $0\le x_{1,n}<x_{2,n}<\dots<x_{n,n}<1$.
We also suppose that there exists $\g\in (0,1]$ such that
\begin{equation}\label{g}
  \min_{k}(x_{k+1,n}-x_{k,n})\ge \frac{\g}{n}\quad\text{for all}\quad n\in \N,
\end{equation}
where $x_{n+1,n}=1+x_{1,n}$.

To formulate the main results, we need the following general conditions on the operators $G_n$, $n\in \N$, and  fixed parameters $1\le p<\infty$ and $s\in \N$:

\begin{equation}\label{c1}
  \|G_n(f)\|_p\le K_1\|f\|_{\ell_p(X_n)},\quad f\in L_p(\T), \quad n\in\N,
\end{equation}

\begin{equation}\label{c1'}
  K_2\|f\|_{\ell_p(X_n)}\le \|G_n(f)\|_p,\quad f\in L_p(\T), \quad n\in\N,
\end{equation}

\begin{equation}\label{c2'}
  \|f-G_n(f)\|_p\le K_3n^{-s}\|f^{(s)}\|_p, \quad f\in W_p^s(\T), \quad n\in \N,
\end{equation}
where $K_i=K_i(p)>0$, $i=1,2$, and $K_3=K_3(s,p)>0$.



\begin{remark}\label{remI}
It follows from~\cite[Lemma~1]{CZ99} that if $G_{2n+1}(f)$ belongs to $\mathcal{T}_n$ and is an interpolation polynomial with nodes $X_{2n+1}$, then ~\eqref{g} is a necessary condition for the validity of~\eqref{c1} and~\eqref{c1'} with $n$ odd. Note also that in this case~\eqref{c1} and~\eqref{c1'} imply~\eqref{c2'} for every $s\in \N$, cf. Example~\ref{exG2} below.
\end{remark}

Let us give several important examples of sampling operators $G_n$ satisfying conditions~\eqref{c1}--\eqref{c2'}.

\begin{example}\label{exG2}
\normalfont 
Let  $\mathcal{Z}_{2n+1}=(z_{k})_{k=0}^{2n}$ be such that
$0\le z_{0}<z_{1}<\dots<z_{2n}<1$.
Consider the interpolation polynomial $L_n^\mathcal{Z}(f)$ such that
$L_n^\mathcal{Z}(f)\in \mathcal{T}_n$ and $L_n^\mathcal{Z}(f)(z_{k})=f(z_{k})$ for each $k=0,\dots,2n$, i.e.,
\begin{equation}\label{LZ}
  L_n^\mathcal{Z}(f)(x)=\sum_{k=0}^{2n}f(z_{k})\frac{\prod\limits_{j\neq k}\sin\frac{x-z_{j}}{2}}{\prod\limits_{j\neq k}\sin\frac{z_{k}-z_{j}}{2}}.
\end{equation}
If follows from~\cite[Theorem~1 and Corollary~1]{CZ99} (see also~\cite{CZ93} and~\cite{MS09}) that if $1<p<\infty$ and
\begin{equation*}
  \max_{k=0,\dots,2n}\Big|z_{k}-\frac{k}{2n+1}\Big|\le \frac{\delta_p}{2n+1}
\end{equation*}
for a particular constant $\d_p>0$, then
\begin{equation}\label{tauC}
  \|f-L_n^\mathcal{Z}(f)\|_p\le \frac{C}n\|f'\|_p,\quad f\in W_p^1(\T),
\end{equation}
and
\begin{equation}\label{le1.1C}
    C_2 \|T_n\|_{\ell_p(\mathcal{Z}_{2n+1})}\le \|T_n\|_p\le C_1 \|T_n\|_{\ell_p(\mathcal{Z}_{2n+1})},\quad T_n\in \mathcal{T}_n,
  \end{equation}
where $C$, $C_1$, and $C_2$ are positive constants depending only on $p$. Using~\eqref{le1.1C},  properties of the error of best one-sided approximation, and simple arguments from~\cite{PQ92} and~\cite{PrXu}, we can easily extend~\eqref{tauC} to the case of $f\in W_p^s(\T)$ for every $s\in \N$. Indeed, let $\ww E_n(f)_p=\|Q_n-q_n\|_p$, where $q_n,Q_n\in \mathcal{T}_n$ and $q_n(x)\le f(x)\le Q_n(x)$ for all $x\in \T$.  Then, by~\eqref{le1.1C} and formulas~\eqref{leJBt.1} and $(e_\tau)$ below, we obtain
\begin{equation*}
  \begin{split}
      \|f-L_n^\mathcal{Z}(f)\|_p&\le \|f-q_n\|_p+\|L_n^\mathcal{Z}(f)-q_n\|_p\\
&\le \|Q_n-q_n\|_p+C_1\|f-q_n\|_{\ell_p(\mathcal{Z}_{2n+1})}\\
&\le \ww E_n(f)_p+C_1\|Q_n-q_n\|_{\ell_p(\mathcal{Z}_{2n+1})}\\
&\le (1+C_1C_2^{-1})\ww E_n(f)_p\le \frac{C_3}{n^s}\|f^{(s)}\|_p,\\
   \end{split}
\end{equation*}
where the constant $C_3$ depends only on $s$ and $p$.
Thus, if $1<p<\infty$, then  $L_n^\mathcal{Z}(f)$ satisfies conditions \eqref{c1} and \eqref{c1'} with respect to the set $\mathcal{Z}_{2n+1}$, and~\eqref{c2'} holds for every $s\in \N$.

An important partial case of~\eqref{LZ} is  the trigonometric Lagrange interpolation polynomial~\eqref{Lag}:
\begin{equation*}
  L_n(f)(x)=\frac1{2n+1}\sum_{k=0}^{2n} f(t_k)D_n(x-t_k),\quad t_k=\frac{k}{2n+1}.
\end{equation*}
\end{example}

\begin{example}\label{exG3}
\normalfont If we replace the Dirichlet kernel $D_n$ in the above formula by other appropriate function, we obtain the so-called quasi-interpolation operators 
\begin{equation*}
  Q_n(f)(x)=\frac1{2n+1}\sum_{k=0}^{2n} f(t_k)\vp_n(x-t_k).
\end{equation*}
It is known that if $\vp_n\in \mathcal{T}_n$ and $\sup_n \|\vp_n\|_1<\infty$, then $Q_n(f)$ satisfies~\eqref{c1} for any $f\in L_p(\T)$, $1\le p<\infty$ (see, e.g.,~\cite[Lemma 17]{KKS20}). Concerning condition~\eqref{c2'}, we have that it holds if, for example, $\vp_n(x)=\sum_{|\ell|\le n}\Phi(\frac{\ell}{2n+1})e^{2\pi{\rm i}\ell x}$, where $\Phi\in C^{s+1}$ in a neighborhood of zero, $\Phi(0)=1$, and $\Phi^{(\nu)}(0)=0$ for all $\nu=1,\dots,s-1$ (see, e.g.,~\cite[Theorem 4.5]{KP21}).
\end{example}



Another operators $G_n$ satisfying~\eqref{c1}--\eqref{c2'} can also be  constructed by means of spline functions.

\begin{example}\label{exG4}
\normalfont For $m\in \Z_+$, $n\in \N$, and $\mathcal{X}_{n}=(x_k)_{k=0}^{n-1}$, where $x_k=\frac kn$, 
we denote by $\mathcal{S}_{m,n}$ the set of all spline functions of
degree $m-1$ with the knots $y_k=x_k+\frac{1+(-1)^m}{4n}$, $k=0,\ldots,n,$ i.e.,
$S\in \mathcal{S}_{m,n}$ if $S\in C^{m-2}[0,1]$ and $S$ is some
algebraic polynomial of degree $m-1$ in each interval
$(y_{k},\,y_{k+1})$, $k=0,\ldots,n-1$.
It is well known that there exists a uniquely determined interpolation operator
$$
I_{m,n}\,:\,L_p(\T)\mapsto \mathcal{S}_{m,n}\quad\text{with}\quad I_{m,n}f(x_k)=f(x_k),\quad k=0,1,\dots,n-1.
$$
It follows from~\cite[Corollary~3 and Lemmas~1 and~2]{PQ92} that, for all $1<p<\infty$, $m\in \Z_+$, and $s\le m+1$, 
\begin{equation*}
  \|f-I_{m,n}(f)\|_p\le \frac{C}{n^s}\|f^{(s)}\|_p,\quad f\in W_p^s(\T),
\end{equation*}
and
\begin{equation*}
    C_2 \|S_n\|_{\ell_p(\mathcal{X}_{n})}\le \|S_n\|_p\le C_1 \|S_n\|_{\ell_p(\mathcal{X}_{n})},\quad S_n\in \mathcal{S}_{m,n},
  \end{equation*}
where $C$, $C_1$, and $C_2$ are positive constants depending only on $p$, $m$, and $s$. Thus, conditions~\eqref{c1}--\eqref{c2'} are fulfilled for the spline interpolation  operator~$I_{m,n}$.
\end{example}

We will formulate our results in terms of the following special Steklov averaged operator defined for an integrable function $f$ and parameters $r\in \N$ and $\d>0$ by
\begin{equation*}
  f_{\d,r}(x)=(-1)^{r+1}\binom{2r}{r}^{-1}\frac2\d\int_{-\d/2}^{\d/2} \sum_{\nu=0}^{r-1}(-1)^\nu \binom{2r}{\nu}f\(x+\frac{r-\nu}{r}t\)dt.
\end{equation*}
Note that in the case $r=1$, $f_{\d,r}$ is the classical Steklov averages, i.e.,
$$
f_{\d,1}(x)=f_\d(x)=\frac1{\d}\int_{-\d/2}^{\d/2} f(x+t)dt
$$
and in the case $r=2$, we have
$$
f_{\d,2}(x)=-\frac1{3\d}\int_{-\d/2}^{\d/2} (f(x+t)-4f(x+t/2))dt.
$$

In what follows, we will constantly use the fact that
\begin{equation}\label{st1}
  \kappa_2(r)\w_{2r}(f,\d)_p\le \|f_{\d,r}-f\|_p\le \kappa_1(r)\w_{2r}(f,\d)_p,
\end{equation}
where $\kappa_1(r)$ and $\kappa_2(r)$ are some positive constants depending only on $r$.
Indeed,  we have
\begin{equation}\label{EQ}
  \begin{split}
     &f_{\d,r}(x)\\
&=(-1)^{r+1}\binom{2r}{r}^{-1}\frac2\d\int_0^{\d/2} \(\sum_{\nu=0}^{2r}(-1)^\nu \binom{2r}{\nu}f\(x+\frac{r-\nu}{r}t\)-(-1)^r\binom{2r}{r}f(x)\)dt\\
&=(-1)^{r+1}\binom{2r}{r}^{-1}\frac2\d\int_0^{\d/2}\D_{t/r}^{2r}f(x-t)dt+f(x).
   \end{split}
\end{equation}
Thus, applying Minkowski's inequality, we get
\begin{equation}\label{EQ+}
\begin{split}
    \binom{2r}{r}\|f_{\d,r}-f\|_p&=\left\|\frac2\d\int_0^{\d/2}\D_{t/r}^{2r}f(\cdot-t)dt\right\|_p\le \frac2\d\int_0^{\d/2}\|\D_{t/r}^{2r}f\|_pdt\\
&\le \w_{2r}(f,\d/(2r))_p \le \w_{2r}(f,\d)_p,
\end{split}
\end{equation}
which implies the upper estimate in~\eqref{st1} with the constant $\kappa_1(r)=\binom{2r}{r}^{-1}$. The lower estimate can be proved using the methods of Fourier multipliers, see, e.g.,~\cite[8.2.5]{TB}, see also~\cite{K17}.

We are interested in the behaviour of  the quantity $\|f_{\d,r}-f\|_{\ell_p(X_n)}$, where
$X_n=(x_{k,n})_{k=1}^{n}$ is such that $0\le x_{1,n}<x_{2,n}<\dots<x_{n,n}<1$ and $x_{n+1,n}=1+x_{1,n}$.
The next result provides an analogue of the upper estimate in~\eqref{st1} for the averaged modulus of smoothness $\tau_{2r}(f,1/n)_p$.

\begin{proposition}\label{cor1+}
Let $f\in L_p(\T)$, $1\le p<\infty$, $r,n\in \N$, and $0<\d\le \min_{k}(x_{k+1,n}-x_{k,n})$.
Then
\begin{equation}\label{EQ0}
    \|f_{\d,r}-f\|_{\ell_p(X_n)}\le \frac{\kappa_1(r)}{(\d n)^{1/p}}\tau_{2r}(f,\d)_p.
\end{equation}
\end{proposition}

\begin{proof}
Denote $x_k=x_{k,n}$.  It follows from~\eqref{EQ} that
\begin{equation}\label{EQ1}
\begin{split}
    |f_{\d,r}(x_k)-f(x_k)|&\le \frac{2\kappa_1(r)}\d\int_0^{\d/2}|\D_{t/r}^{2r}f(x_k-t)|dt\\
&=\frac{2\kappa_1(r) }\d\int_{x_k-\d/2}^{x_k}|\D_{\frac{x_k-t}r}^{2r}f(t)|dt.
\end{split}
\end{equation}
Let $t\in [x_k-\d/2, x_k]$ and $\t\in [x_k-\d/2,x_k+\d/2]$. Then, taking into account that $[x_k-\d/2, x_k+\d/2]\subset [\t-\d, \t+\d]$, we get
\begin{equation*}
  \begin{split}
      \big|\D_{\frac{x_k-t}r}^{2r}f(t)\big|\le \w_{2r}(f,x_k,\d/(2r))\le \w_{2r}(f,\t,\d/r)
  \end{split}
\end{equation*}
and, therefore,  by~\eqref{EQ1} and $(a_\tau)$ below, we obtain
\begin{equation*}
  \begin{split}
      \frac1n \sum_{k=1}^n|f_{\d,r}(x_k)-f(x_k)|^p&\le \frac{\kappa_1(r)^p}{\d n} \sum_{k=1}^n\int_{x_k-\d/2}^{x_k+\d/2}\(\w_{2r}(f,\t,\d/r)\)^p\,d\t\\
&\le \frac{\kappa_1(r)^p}{\d n} \tau_{2r}(f,\d)_p^p,
  \end{split}
\end{equation*}
which implies~\eqref{EQ0}.
\end{proof}



\section{Auxiliary results}


\subsection{Integral moduli of smoothness and the error of best approximation}
Let us recall several basic properties of the integral moduli of smoothness (see, e.g.,~\cite[Ch.~2]{DL}, \cite[Ch.~4]{TB}).
For  $f,g\in L_p(\T)$, $1\le p<\infty$, $\d>0$, and $r\in \N$, we have 
\begin{itemize}
\item[$(a_\w)$]
$ \w_r(f,\d)_p$ is a non-negative non-decreasing function of $\d$ such that
$\lim\limits_{\d\to 0+} \w_r(f,\delta)_p=0;$ 
\item[$(b_\w)$]
$\w_r(f+g,\d)_p\le \w_r(f,\d)_p+\w_r(g,\d)_p;$

 \item[$(c_\w)$]  $ \w_{r+1}(f,\d)_p\le 2\w_r(f,\d)_p$;

\item[$(d_\w)$]
for $\l>0$,
        $$\w_r(f,\lambda \delta)_p\le (1+\l)^r  \w_r(f,\d)_p;$$

\item[$(e_\w)$] $\w_r(f,\d)_p\le \d^r \|f^{(r)}\|_p$ for all $f\in W_p^r(\T)$.

%

\end{itemize}

In the next two lemmas, we recall the classical direct and inverse approximation theorems and one result on the exact order of the error of best approximation.

\begin{lemma}\label{leJB} {\normalfont (See, e.g.,~\cite[Ch.~7]{DL}.)}
  Let $f\in L_p(\T)$, $1\le p<\infty$, and $r,n\in \N$. Then
\begin{equation}\label{leJB.1}
  E_{n-1}(f)_p\le c_1(r)\w_r(f,1/n)_p
\end{equation}
and
\begin{equation}\label{leJB.2}
  \w_r(f,1/n)_p\le \frac{c_2(r)}{n^r}\sum_{\nu=0}^n(\nu+1)^{r-1}E_\nu(f)_p.
\end{equation}
\end{lemma}


\begin{lemma}\label{lemR} {\normalfont (See~\cite{R94}.)}
  Let $f\in L_p(\T)$, $1\le p<\infty$, and $r\in \N$. There exists a positive constant $G$ such that
$$
\w_r(f,1/n)_p\le GE_n(f)_p,\quad n\in \N,
$$
if and only if there exists a positive constant $F$ such that
$$
\w_r(f,\d)_p\le F\w_{r+1}(f,\d)_p,\quad \d>0.
$$
\end{lemma}

We also need the Nikol'skii--Stechkin--Boas-type inequality (see, e.g.,~\cite[p.~214 and p.~251]{timan}), which
states that, for $T_n\in \mathcal{T}_n$, $r\in\N$, and $1\le p<\infty$:
\begin{equation}\label{NS-}
\| T_n^{(r)}  \|_p\le
\(\frac{n}{2\sin\frac{n\delta}{2}}\)^r \| \D_\delta^r  T_n\|_{p},\quad 0<\delta\le\frac\pi n.
\end{equation}
In particular, if $T_n\in \mathcal{T}_n$ is a polynomial of the best approximation to $f$ in $L_p(\T)$, i.e., $\|f-T_n\|_p=E_n(f)_p$, then by direct inequality~\eqref{leJB.1}, $(b_\w)$, and $(c_\w)$, we have 
\begin{equation}\label{NS}
  \| T_n^{(r)}\|_p\le c_3(r)n^r\w_r(f,1/n)_p,\quad n\in \N.
\end{equation}

\subsection{The averaged moduli of smoothness and the error of best one-sided approximation}

The following basic properties of the averaged moduli of smoothness can be found in~\cite[Ch.~1]{SP}.
For  $f,g\in B(\T)$, $1\le p<\infty$, $\d>0$, and $r\in \N$, we have 
\begin{itemize}
\item[$(a_\tau)$]
$ \tau_r(f,\d)_p$ is a non-negative non-decreasing function of $\d$ and $\lim\limits_{\d\to 0+} \tau_r(f,\delta)_p=0$ if $f$ is a  Riemann integrable function on $\T$; 
\item[$(b_\tau)$]
$\tau_r(f+g,\d)_p\le \tau_r(f,\d)_p+\tau_r(g,\d)_p;$

 \item[$(c_\tau)$]  $ \tau_{r+1}(f,\d)_p\le 2\tau_r\(f,\frac{r+1}{r}\d\)_p$;

\item[$(d_\tau)$]
for $\l>0$,
        $$\tau_r(f,\lambda \delta)_p\le (2(1+\l))^{r+1}  \tau_r(f,\d)_p,\quad r\ge 2,$$
        $$\tau_1(f,\lambda \delta)_p\le (1+\l)  \tau_1(f,\d)_p;$$
\item[$(e_\tau)$] $\tau_r(f,\d)_p\le c_4(r)\d^r \|f^{(r)}\|_p$ for each $f\in W_p^r(\T)$.
\end{itemize}

The relations between $\tau_r(f,\d)_p$ and $\w_r(f,\d)_p$ are given in the following two properties:

\begin{itemize}

\item[$(f_\tau)$]   $\w_r(f,\d)_p\le \tau_r(f,\d)_p\le (2\pi)^{1/p} \w_r(f,\d)_\infty$;

\item[$(g_\tau)$] if $f\in C(\T)$, then  $$
                   \tau_r(f,\d)_p\le c(r,p)\d^{1/p}\int_{0}^\d \frac{\w_r(f,\d)_p}{t^{1/p+1}}dt.
                    $$
\end{itemize}

\bigskip

The next lemma is an analogue of Lemma~\ref{leJB} for the averaged moduli of smoothness and the error of best one-sided approximation.

\begin{lemma}\label{leJBt}
  Let $f\in L_p(\T)$, $1\le p<\infty$, and $r,n\in \N$. Then
\begin{equation}\label{leJBt.1}
  \ww E_{n-1}(f)_p\le c_1(r)\tau_r(f,1/n)_p
\end{equation}
and
\begin{equation}\label{leJBt.2}
  \tau_r(f,1/n)_p\le \frac{c_2(r)}{n^r}\sum_{\nu=0}^n(\nu+1)^{r-1}\ww E_\nu(f)_p.
\end{equation}
\end{lemma}

We also need the following analogue of Lemma~\ref{lemR} for the averaged moduli of smoothness.

\begin{lemma}\label{lemR+}
  Let $f\in L_p(\T)$, $1\le p<\infty$. There exists a constant $G>0$ such that
\begin{equation}\label{lemR+.1}
  \tau(f,1/n)_p\le G\ww E_n(f)_p,\quad n\in \N,
\end{equation}
if and only if there exists a constant $F>0$ such that
\begin{equation}\label{lemR+.2}
\tau(f,\d)_p\le F\tau_{2}(f,\d)_p,\quad \d>0.
\end{equation}
\end{lemma}

\begin{proof}
First, we prove the sufficiency. By~\eqref{lemR+.2} and properties $(c_\tau)$ and
$(d_\tau)$, we get
\begin{equation}\label{lemR+.3}
      \tau_2(f,\l h)_p\le 2(1+2\l)F\tau_2(f,h)_p,\quad \l>0.
\end{equation}
Then, by the direct inequality~\eqref{leJBt.1} and~\eqref{lemR+.3},
we have
\begin{equation}\label{rathInv}
\begin{split}
      \frac{1}{n^{2}}\sum_{\nu=0}^n(\nu+1)&\ww E_\nu(f)_p\le
      \frac{c_1}{n^{2}}\sum_{\nu=0}^n(\nu+1)\tau_2\(f,\frac
      {1}{\nu+1}\)_p\le\\
      &\le\frac{2c_1F}{n^{2}}\tau_2(f,n^{-1})_p\sum_{\nu=0}^n(\nu+1)\(1+\frac{2n}{\nu+1}\)\le
      14c_1F\tau_2(f,n^{-1})_p.
\end{split}
\end{equation}
Let $m\in \N$. Applying~\eqref{leJBt.2} and~\eqref{rathInv}, we obtain 
\begin{equation*}
\begin{split}
       \tau_2\left(f,\frac{1}{mn} \right)_p
       &\le\frac{c_2}{(mn)^{2}}\(\sum_{\nu=n+1}^{mn}(\nu+1)\ww E_{\nu}(f)_p+
       \sum_{\nu=0}^n(\nu+1)\ww E_{\nu}(f)_p \)\\
       &\le \frac{c_2}{(mn)^{2}}\(\sum_{\nu=n+1}^{mn}
      (\nu+1)\ww E_{\nu}(f)_p+{14c_1F}{n^{2}}
      \tau_2\left(f,\frac 1n\right)_p\),
\end{split}
\end{equation*}
which implies that
\begin{equation*}
      \sum_{\nu=n+1}^{mn}(\nu+1)\ww E_{\nu}(f)_p\ge
      \frac{(mn)^{2}}{c_2}\tau_2\left(f,\frac{1}{mn}\right)_p-14c_1Fn^{2}\tau_2\left(f,\frac
     1n\right)_p.
\end{equation*}
Then, using the monotonicity of $\ww E_n(f)_p$ and again inequality~\eqref{lemR+.3}, we get
\begin{equation}\label{rathInv+}
      \ww E_n(f)_p\sum_{\nu=n+1}^{mn}(\nu+1)\ge n^{2}\(\frac{m}{6c_2F}-14c_1F\)
      \tau_2\left(f,n^{-1}\right)_p.
\end{equation}
Thus, choosing an appropriate $m$ in~\eqref{rathInv+}, we can find a positive constant $c$ independent of $n$ such that
\begin{equation*}
      \ww E_n(f)_p\ge c\tau_2\left(f,n^{-1}\right)_p.
\end{equation*}
From the last inequality and~\eqref{lemR+.2}, we derive~\eqref{lemR+.1}.

   The necessity follows directly from Jackson's-type inequality~\eqref{leJBt.1}.
\end{proof}

\subsection{Marcinkiewicz-Zygmund type inequalities}

The following Marcinkiewicz-Zygmund type inequality was established in~\cite{O86}, see also~\cite{LMN}.

\begin{lemma}\label{leMZ1}
  Let $1\le p<\infty$, $m\in \N$, and let $(x_k)_{k=1}^m$ be such that
  $0\le x_1<x_2<\dots<x_m<1$ and $\d=\min_k(x_{k+1}-x_k)$, where $x_{n+1}=1+x_1$. Then, for each $T_n\in \mathcal{T}_n$, $n\in \N$, we have
\begin{equation*}
  \sum_{k=1}^m |T_n(x_k)|^p\le (p+1)\frac e2 \(2n+\frac1\delta\)\int_0^1 |T_n(x)|^p dx.
\end{equation*}
\end{lemma}


\section{Main results}

\subsection{Direct and inverse inequalities}


\begin{theorem}\label{th1+}
  Let $f\in L_p(\T)$, $1\le p<\infty$, and $r,n\in \N$. Suppose $G_n(f)$ satisfies~\eqref{c1} and~\eqref{c2'} with $s\le 2r$. Then
  \begin{equation}\label{th1.1'}
     \|f-G_n(f)\|_p\le K_1\|f_{\g/n,\,r}-f\|_{\ell_p(X_n)}+{C_1}\w_s(f,1/n)_p.
  \end{equation}
If, additionally, \eqref{c1'} holds, then
  \begin{equation}\label{th1.2'}
     K_2\|f_{\g/n,\,r}-f\|_{\ell_p(X_n)}-{C_2}\w_s(f,1/n)_p\le \|f-G_n(f)\|_p.
  \end{equation}
Here, $C_1$ and $C_2$ are some positive constants independent of $f$ and $n$.
\end{theorem}

\begin{proof}
Let $T_n\in \mathcal{T}_n$ be such that $\|f-T_n\|_p=E_n(f)_p$. Then, using~\eqref{leJB.1}, \eqref{c2'}, \eqref{c1}, and~\eqref{NS}, we have
\begin{equation}\label{th1.3'}
\begin{split}
    \|f-G_n(f)\|_p&\le \|f-T_n\|_p+\|T_n-G_n(T_n)\|_p+\|G_n(T_n-f)\|_p\\
&\le c_1\w_s(f,1/n)_p+K_3n^{-s}\|T_n^{(s)}\|_p+K_1\|f-T_n\|_{\ell_p(X_n)}\\
&\le (c_1+c_3K_3)\w_s(f,1/n)_p+K_1\|f-T_n\|_{\ell_p(X_n)}.
\end{split}
\end{equation}
Let $\d\in (0,\min_{k}(x_{k+1}-x_{k})]$, where $x_k=x_{k,n}$, $k=1,\dots,n+1$. We estimate $\|f-T_n\|_{\ell_p(X_n)}$ as follows:
\begin{equation}\label{th1.4'}
  \begin{split}
      \|f-T_n\|_{\ell_p(X_n)}\le \|f-f_{\d,r}\|_{\ell_p(X_n)}+\|f_{\d,r}-(T_n)_{\d,r}\|_{\ell_p(X_n)}+\|(T_n)_{\d,r}-T_n\|_{\ell_p(X_n)}.
  \end{split}
\end{equation}
By H\"older's inequality and~\eqref{leJB}, we obtain
\begin{equation}\label{th1.5'}
  \begin{split}
      \|f_{\d,r}-&(T_n)_{\d,r}\|_{\ell_p(X_n)}^p\\
&\le \frac1n\sum_{k=1}^n\(\sum_{\nu=0}^{r-1}\frac2{\d}\int_{-\d/2}^{\d/2}|f(x_k+\tfrac{r-\nu}{r}t)-T_n(x_k+\tfrac{r-\nu}{r}t)|dt\)^p\\
&=\frac1n\sum_{k=1}^n\(\sum_{\nu=0}^{r-1}\(\frac{r}{r-\nu}\)\frac2{\d}\int_{-(r-\nu)\d/2r}^{(r-\nu)\d/2r}|f(x_k+t)-T_n(x_k+t)|dt\)^p\\
&\le \frac1n\sum_{k=1}^n\(\frac{2r^2}{\d}\int_{x_k-\d/2}^{x_k+\d/2}|f(t)-T_n(t)|dt\)^p\\
&\le \frac{2^pr^{2p}}{n\d}\sum_{k=1}^n\int_{x_k-\d/2}^{x_k+\d/2}|f(t)-T_n(t)|^pdt\\
&\le \frac{2^pr^{2p}}{n\d}\|f-T_n\|_p^p\le \frac{2^pr^{2p}c_1^p}{n\d}\w_s(f,1/n)_p^p.
  \end{split}
\end{equation}
To estimate $\|(T_n)_{\d,r}-T_n\|_{\ell_p(X_n)}$, we use Lemma~\ref{leMZ1}, the upper estimate in~\eqref{st1}, $(e_\w)$, \eqref{NS},
and~$(c_\w)$:
\begin{equation}\label{th1.6'}
  \begin{split}
    \|(T_n)_{\d,r}-T_n\|_{\ell_p(X_n)}&\le c(n\d)^{-1/p}\|(T_n)_{\d,r}-T_n\|_{p}\le c(n\d)^{-1/p}\kappa_1\w_{2r}(T_n,\d)_p\\
 &\le c\kappa_1\d^{2r}(n\d)^{-1/p}\|T_n^{(2r)}\|_p\le c\kappa_1c_3(n\d)^{-1/p}\w_{2r}(f,1/n)_p\\
 &\le c\kappa_1c_32^{2r-s}(n\d)^{-1/p}\w_{s}(f,1/n)_p,
  \end{split}
\end{equation}
where $c=(\frac32(p+1)e)^{1/p}$. Thus, combining~\eqref{th1.4'}, \eqref{th1.5'}, and~\eqref{th1.6'}, we have
\begin{equation}\label{th1.7'}
  \begin{split}
      \|f-T_n\|_{\ell_p(X_n)}\le \|f-f_{\d,r}\|_{\ell_p(X_n)}+\(c\kappa_1c_32^{2r-s}+{2r^2c_1}\)(n\d)^{-1/p}\w_s(f,1/n)_p,
  \end{split}
\end{equation}
which together with~\eqref{th1.3'} implies~\eqref{th1.1'} when $\d=\g/n$.

\smallskip

Now we prove~\eqref{th1.2'}. Inequality~\eqref{c1'} yields
\begin{equation}\label{th1.8'}
\begin{split}
    K_2\|f_{\d,r}-f\|_{\ell_p(X_n)}&\le \|G_n(f_{\d,r}-f)\|_p\\
  &\le \|f-f_{\d,r}\|_p+\|f_{\d,r}-G_n(f_{\d,r})\|_p+\|f-G_n(f)\|_p.\\
\end{split}
\end{equation}
By~\eqref{st1}, properties $(c_\w)$ and $(a_\w)$, and the fact that $\d\le 1/n$, we get
\begin{equation}\label{th1.9'}
  \|f-f_{\d,r}\|_p\le \kappa_1\w_{2r}(f,\d)_p\le 2^{2r-s}\kappa_1\w_s(f,1/n)_p.
\end{equation}
Next, using Minkowski's inequality, \eqref{c2'}, and~\eqref{c1}, we obtain
\begin{equation}\label{th1.10'}
\begin{split}
\|f_{\d,r}-&G_n(f_{\d,r})\|_p\\
&\le \|f_{\d,r}-(T_n)_{\d,r}\|_p+\|(T_n)_{\d,r}-G_n((T_n)_{\d,r})\|_p+\|G_n(f_{\d,r}-(T_n)_{\d,r})\|_p\\
&\le \|f-T_n\|_p+K_3n^{-s}\|(T_n)_{\d,r}^{(s)}\|_p+K_1\|f_{\d,r}-(T_n)_{\d,r}\|_{\ell_p(X_n)}.
\end{split}
\end{equation}
Applying again Minkowski's inequality and~\eqref{NS}, we have
\begin{equation}\label{th1.11'}
  n^{-s}\|(T_n)_{\d,r}^{(s)}\|_p\le n^{-s}\|T_n^{(s)}\|_p\le c_3\w_s(f,1/n)_p.
\end{equation}
Thus, \eqref{th1.10'} together with~\eqref{leJB.1}, \eqref{th1.11'}, and~\eqref{th1.5'} yields
\begin{equation}\label{th1.12'}
  \|f_{\d,r}-G_n(f_{\d,r})\|_p\le \(c_1+c_3K_3+\frac{2r^{2}c_1K_1}{(\d n)^{1/p}}\)\w_s(f,1/n)_p.
\end{equation}
Finally, combining~\eqref{th1.8'}, \eqref{th1.9'}, \eqref{th1.12'}, and taking $\d=\g/n$, we arrive at~\eqref{th1.2'}.
\end{proof}

As a simple corollary of Theorem~\ref{th1}, we have the following convergence criteria of the sampling operators $G_n$, $n\in \N$, in $L_p(\T)$.

\begin{corollary}\label{th2}
  Let $1\le p<\infty$ and $r\in \N$. Suppose that $G_n$, $n\in \N$, satisfy conditions~\eqref{c1}, \eqref{c1'}, and~\eqref{c2'} with $s\le 2r$. Then $(G_n(f))_{n\in \N}$ converges to  $f$ in $L_p(\T)$ if and only if
$$
\|f_{\g/n,\,r}-f\|_{\ell_p(X_n)}\to 0\quad\text{as}\quad n\to \infty.
$$
\end{corollary}

\begin{proof}
  The assertion follows directly from Theorem~\ref{th1} by taking into account that $\w_{s}(f,\d)_p\to 0$  as $\d\to 0$ in view of~$(a_\w)$.
\end{proof}


Using Theorem~\ref{th1+}, it is not difficult to find the exact order of convergence of $G_n(f)$ for functions $f$ satisfying certain special conditions.


\begin{corollary}\label{cor1}
 Let $f\in L_p(\T)$, $1\le p<\infty$, and $r\in \N$. Suppose  $G_n(f)$, $n\in \N$,  satisfy conditions~\eqref{c1}, \eqref{c1'}, and~\eqref{c2'} with $s\le 2r$. If 
  $$
  \w_s(f,1/n)_p=o\(\|f_{\g/n,\,r}-f\|_{\ell_p(X_n)}\),
  $$
  then
  $$
  \|f-G_n(f)\|_p\sim \|f_{\g/n,\,r}-f\|_{\ell_p(X_n)}.
  $$
\end{corollary}

The next theorem is an analogue of Lemmas~\ref{lemR} and~\ref{lemR+} in the case of approximation by sampling operators.

\begin{theorem}\label{th1RG}
  Let $f\in L_p(\T)$, $1\le p<\infty$, and $r,n\in \N$. Suppose that $G_n$, $n\in \N$, satisfy conditions~\eqref{c1}, \eqref{c1'}, and~\eqref{c2'} with $s\le 2r$.  If $G_n(f)\in \mathcal{T}_n$ and there exists a constant $K>0$ such that
$$
\w_s(f,h)_p\le K\w_{s+1}(f,h)_p,\quad h>0,
$$
then
  \begin{equation*}
     \|f-G_n(f)\|_p\asymp\|f_{\g/n,\,r}-f\|_{\ell_p(X_n)}+\w_s(f,1/n)_p,
  \end{equation*}
where $\asymp$ is a two-sided inequality with positive constants depending only on $K$, $p$, $r$, and $s$.
\end{theorem}

\begin{proof}
In view of Theorem~\ref{th1+}, it is enough to prove only the lower estimate. From~\eqref{th1.2'}, we derive
\begin{equation}\label{th3.2}
\begin{split}
    \|f_{\g/n,\,r}-f\|_{\ell_p(X_n)}&+\w_s(f,1/n)_p\\
  &\le \frac1{K_2}\(\|f-G_n(f)\|_p+(K_2+C_2)\w_s(f,1/n)_p\).
\end{split}
\end{equation}
It remains to apply Lemma~\ref{lemR} and to take into account that $E_n(f)_p\le \|f-G_n(f)\|_p$.
\end{proof}

The next result provides an analogue of Bernstein's type inverse inequalities~\eqref{leJB.2} and~\eqref{leJBt.2} for sampling operators.
To formulate it, we need the following additional assumption on $G_n(f)$ and $s\in \N$:
\begin{equation}\label{c4}
  \|(G_{2^{\nu}}(f)-G_{2^{\nu-1}}(f))^{(s)}\|_p\le K_{4}(s,p) 2^{s\nu}\|G_{2^{\nu}}(f)-G_{2^{\nu-1}}(f)\|_p, \quad \nu\in\N.
\end{equation}

\begin{theorem}\label{th3}
  Let $f\in L_p(\T)$, $1\le p<\infty$, and $r,n\in \N$. Suppose that $G_n$, $n\in \N$, satisfy conditions~\eqref{c1}, \eqref{c1'}, \eqref{c2'} with $s\le 2r$, and~\eqref{c4}. Then
  \begin{equation}\label{th3.1}
  \begin{split}
         \|f_{\g/n,\,r}&-f\|_{\ell_p(X_n)}+\w_s(f,1/n)_p\\
&\le C\bigg(\|f-G_n(f)\|_p
         +\frac1{n^s}\sum_{k=0}^{[\log_2 n]} 2^{sk}\(\|f-G_{2^{k}}(f)\|_p+\|f-G_{2^{k-1}}(f)\|_p\)\bigg),
  \end{split}
  \end{equation}
  where the constant $C$ does not depend on $f$ and $n$.
\end{theorem}

\begin{proof}
Denote $m=[\log_2 n]$, $G_{1/2}=0$, and $G_k=G_k(f)$, $k\in \N$. By $(b_\w)$ and $(c_\w)$, we have
\begin{equation}\label{th3.2+}
  \w_s(f,n^{-1})_p\le 2^s\|f-G_{2^m}\|_p+\w_s(G_{2^m},n^{-1})_p.
\end{equation}
Using the representation
$$
G_{2^m}=\sum_{\nu=0}^m(G_{2^\nu}-G_{2^{\nu-1}}),
$$
properties $(b_\w)$, $(e_\w)$, and~\eqref{c4}, we obtain
\begin{equation}\label{th3.2++}
  \begin{split}
      \w_s(G_{2^m},n^{-1})_p&\le \sum_{\nu=0}^m\w_s(G_{2^\nu}-G_{2^{\nu-1}},n^{-1})_p\\
      &\le \frac1{n^s}\sum_{\nu=0}^m\|(G_{2^\nu}-G_{2^{\nu-1}})^{(s)}\|_p\le \frac{K_{4}}{n^s}\sum_{\nu=0}^m2^{s\nu}\|G_{2^\nu}-G_{2^{\nu-1}}\|_p\\
      &\le \frac{K_{4}}{n^s}\sum_{\nu=0}^m2^{s\nu}(\|f-G_{2^\nu}\|_p+\|f-G_{2^{\nu-1}}\|_p).
   \end{split}
\end{equation}
Thus, by~\eqref{th3.2+} and~\eqref{th3.2++}, we have
\begin{equation}\label{th3.3}
  \w_s(f,n^{-1})_p\le\frac{K_{4}+4^s}{n^s}\sum_{k=0}^m2^{sk}(\|f-G_{2^k}\|_p+\|f-G_{2^{k-1}}\|_p).
\end{equation}
Finally, combining~\eqref{th3.2} and~\eqref{th3.3}, we get~\eqref{th3.1}.
\end{proof}

\begin{remark}\label{rem1}
In the case $G_\nu\in \mathcal{T}_\nu$, $\nu\in \N$, inequality~\eqref{th3.1} is simplified as follows:
  \begin{equation*}
  \begin{split}
         \|f_{\g/n,\,r}-f\|_{\ell_p(X_n)}+\w_s(f,1/n)_p\le C\bigg(\|f-G_n(f)\|_p+\frac1{n^s}\sum_{\nu=0}^n(\nu+1)^{s-1} \|f-G_\nu(f)\|_p\bigg).
  \end{split}
  \end{equation*}
The above estimate follows directly from~\eqref{th3.1} and~\eqref{leJB.2}.
\end{remark}

As an application of Theorems~\ref{th1+} and~\ref{th3}, we obtain the following result. 
\begin{corollary}\label{cor2}
Let $f\in L_p(\T)$, $1\le p<\infty$, $r,s\in \N$, $s\le 2r$, and $\a\in (0,s)$.  Suppose that $G_n(f)$, $n\in \N$, satisfy the conditions of Theorem~\ref{th3}. Then the following properties are equivalent:
    \begin{enumerate}
    \item[$(i)$]  $\|f-G_n(f)\|_p=\mathcal{O}(n^{-\a})$,
    \item[$(ii)$]  $\|f_{\g/n,\,r}-f\|_{\ell_p(X_n)}+\w_s(f,1/n)_p=\mathcal{O}(n^{-\a})$.
  \end{enumerate}
\end{corollary}

\subsection{$K$-functionals, realizations, and smoothness of sampling operators}

Recall that for a given function $f\in L_p(\T)$, $1\le p<\infty$, $s\in \N$, and $\d>0$, the Peetre  $K$-functional is defined by
\begin{equation}\label{k0}
  K_s(f,\d)_p=\inf_{g\in W_p^s(\T)}(\|f-g\|_p+\d^s\|g^{(s)}\|_p).
\end{equation}
It is well known (see, e.g.,~\cite[Ch.~6]{DL}) that
\begin{equation}\label{k1}
  K_s(f,\d)_p\asymp \w_s(f,\d)_p,\quad f\in L_p(\T), \quad \d>0,
\end{equation}
where $\asymp$ is a two-sided inequality with positive constants independent of $f$ and $\d$. For related results for the averaged modulus of smoothness $\tau_s(f,\d)_p$ see, e.g.,~\cite{P83}, \cite{P84}. In this section, we establish an analogue of~\eqref{k1} for the following "semi-discrete" modification of the Peetre $K$-functional:
\begin{equation*}
  \mathcal{K}_s(f,X_n)_p:=\inf_{g\in W_p^s(\T)}(\|f-g\|_{\ell_p(X_n)}+\|f-g\|_p+n^{-s}\|g^{(s)}\|_p).
\end{equation*}

\begin{theorem}\label{thKw}
  Let $f\in L_p(\T)$, $1\le p<\infty$, $r,s\in \N$, $s\le 2r$, and $n\in \N$. Then
  \begin{equation}\label{thKw.1}
     \mathcal{K}_s(f,X_n)_p\asymp \|f_{\g/n,\,r}-f\|_{\ell_p(X_n)}+\w_s(f,1/n)_p,
  \end{equation}
where $\asymp$ is a two-sided inequality with positive constants independent of $f$ and $n$.
\end{theorem}

\begin{proof}
  First, we prove the upper estimate. Let $T_n\in \mathcal{T}_n$ be such that $\|f-T_n\|_p=E_n(f)_p$. Then, by the definition of $\mathcal{K}_s(f,X_n)_p$ and~\eqref{leJB.1}, \eqref{NS}, we have
\begin{equation*}
\begin{split}
    \mathcal{K}_s(f,X_n)_p&\le \|f-T_n\|_{\ell_p(X_n)}+\|f-T_n\|_p+n^{-s}\|T_n^{(s)}\|_p\\
&\le \|f-T_n\|_{\ell_p(X_n)}+(c_1+c_3)\w_s(f,n^{-1})_p.
\end{split}
\end{equation*}
Thus, applying inequality~\eqref{th1.7'}, we arrive at the upper estimate in~\eqref{thKw.1}.

Consider the lower estimate. Let $g\in W_p^s(\T)$ and $\d=\g/n$. Then
\begin{equation}\label{thKw.3}
  \begin{split}
     \|f_{\d,r}-f\|_{\ell_p(X_n)}\le \|f_{\d,r}-g_{\d,r}\|_{\ell_p(X_n)}+\|g_{\d,r}-g\|_{\ell_p(X_n)}+\|g-f\|_{\ell_p(X_n)}.
  \end{split}
\end{equation}
By the same arguments as in~\eqref{th1.5'}, we obtain
\begin{equation}\label{thKw.4}
  \begin{split}
      \|f_{\d,r}-g_{\d,r}\|_{\ell_p(X_n)}\le {2r^{2}}{(n\d)^{-1/p}}\|f-g\|_p= {2r^{2}}{\g^{-1/p}}\|f-g\|_p.
  \end{split}
\end{equation}
Now, we estimate $\|g_{\d,r}-g\|_{\ell_p(X_n)}$. Using Proposition~\ref{cor1+} as well as properties~$(c_\tau)$, $(e_\tau)$, and $(d_\tau)$, we get
\begin{equation}\label{thKw.5}
  \begin{split}
      \|g_{\d,r}-g\|_{\ell_p(X_n)}\le k_1(r)\g^{-1/p}\tau_{2r}(g,n^{-1})_p\le c_5\g^{-1/p}n^{-s}\|g^{(s)}\|_p,
  \end{split}
\end{equation}
where the constant $c_5$ depends only on $r$ and $s$. Thus, combining~\eqref{thKw.3}--\eqref{thKw.5}, we obtain
\begin{equation}\label{thKw.6}
  \begin{split}
     \|f_{\d,r}-f\|_{\ell_p(X_n)}\le \|f-g\|_{\ell_p(X_n)}+{2r^{2}}{\g^{-1/p}}\|f-g\|_p+c_5\g^{-1/p}n^{-s}\|g^{(s)}\|_p.
  \end{split}
\end{equation}

Next, \eqref{k1} yields
\begin{equation*}
  \w_s(f,1/n)_p\le 2^s\|f-g\|_p+n^{-s}\|g^{(s)}\|_p,
\end{equation*}
which together with~\eqref{thKw.6} implies that
\begin{equation*}
  \begin{split}
     \|f_{\d,r}-&f\|_{\ell_p(X_n)}+\w_s(f,1/n)_p\\
&\le \|g-f\|_{\ell_p(X_n)}+({2r^{2}}{\g^{-1/p}}+2^s)\|f-g\|_p+(c_5\g^{-1/p}+1)n^{-s}\|g^{(s)}\|_p.
   \end{split}
\end{equation*}
It remains to take the infimum over all $g\in W_p^s(\T)$.
\end{proof}

In addition to~\eqref{k1}, it is known (see, e.g.,~\cite{HI90}) that the modulus of smoothness is equivalent to the so-called realization of the $K$-functional:
\begin{equation}\label{Re}
  \w_s(f,1/n)_p\asymp \|f-T_n(f)\|_p+n^{-s}\|(T_n(f))^{(s)}\|_p,\quad n\in \N,
\end{equation}
where $T_n(f)$ is a trigonometric polynomial of degree at most $n$ such that $\|f-T_n(f)\|_p\lesssim \w_s(f,1/n)_p$.  
This holds, for example, for polynomials of near best approximation, de la Vall\'ee Poussin means, corresponding Riesz means, etc.
For various applications of realizations of the $K$-functionals see e.g.,~\cite{HI90}, \cite{KT20}--\cite{KT21}.
Below, we give an analogue of equivalence~\eqref{Re} for the sampling operator $G_n$.
To make the results more transparent, we additionally assume that
\begin{equation}\label{c3}
  n^{-s}\|(G_n(f))^{(s)}\|_p\le K_5(s,p)\,\w_s\(G_n(f),n^{-1}\)_p, \quad f\in L_p(\T), \quad n\in\N.
\end{equation}
Note that if $G_n(f)$ belongs to $\mathcal{T}_n$, then~\eqref{c3} follows directly from the Nikolskii-Stechkin-Boas inequality~\eqref{NS-}.
Estimate~\eqref{c3} also holds if $G_n(f)$ is a spline function in $\mathcal{S}_{m,n}$, where $m\ge s+1$ (see, e.g.,~\cite{HY95} and~\cite[Ch.~5]{DL}).

\begin{theorem}\label{th4} Let $f\in L_p(\T)$, $1\le p<\infty$, and $r,n\in \N$. Suppose that $G_n$, $n\in \N$, satisfy conditions~\eqref{c1}, \eqref{c1'}, \eqref{c2'} with $s\le 2r$, and~\eqref{c3}.
Then
\begin{equation}\label{th4.1}
  \|f_{\g/n,\,r}-f\|_{\ell_p(X_n)}+\w_s(f,1/n)_p\asymp \|f-G_n(f)\|_p+n^{-s}\|(G_n(f))^{(s)}\|_p,
\end{equation}
where $\asymp$ is a two-sided inequality with positive constants independent of $f$ and $n$.
\end{theorem}

\begin{proof}
The estimate from above follows from~\eqref{th3.2} and the standard properties of moduli of smoothness, c.f.~\eqref{Re}.
To prove the estimate from below, we note that by~\eqref{c3}, $(b_\w)$, and $(c_\w)$,
  \begin{equation}\label{th4.4}
    \begin{split}
      n^{-s}\|(G_n(f))^{(s)}\|_p\le K_5\w_s(G_n(f),n^{-1})_p\le K_5(2^{s}\|f-G_n(f)\|_p+\w_s(f,n^{-1})_p).
    \end{split}
  \end{equation}
  Thus, applying~\eqref{th1.1'} to~\eqref{th4.4}, we obtain
    \begin{equation*}
    \begin{split}
       \|f&-G_n(f)\|_p+n^{-s}\|(G_n(f))^{(s)}\|_p\\
       &\le (1+2^sK_5)\|f-G_n(f)\|_p+K_5\w_s(f,n^{-1})_p\\
       &\le K_1(1+2^sK_5)\|f_{\g/n,\,r}-f\|_{\ell_p(X_n)}+(C_1(1+2^sK_5)+K_5)\w_s(f,n^{-1})_p,
    \end{split}
  \end{equation*}
  which implies the estimate from below in~\eqref{th4.1}.
\end{proof}

It follows from~\eqref{th4.1} that
\begin{equation*}
 n^{-s}\|(G_n(f))^{(s)}\|_p \le C(\|f_{\g/n,r}-f\|_{\ell_p(X_n)}+\w_s(f,1/n)_p),
\end{equation*}
where $C$ does not depend on $f$ and $n$.
In the next theorem, we establish a converse inequality in some sense.

\begin{theorem}\label{th5} Let $f\in L_p(\T)$, $1\le p<\infty$, and $r,n\in \N$. Suppose that $G_n$, $n\in \N$, satisfy conditions~\eqref{c1}, \eqref{c1'}, \eqref{c2'} with $s\le 2r$, and~\eqref{c3}. Assume additionally that $(G_k(f))_{k\in \N}$ converges to $f$  in $L_p(\T)$, $G_n(f)(\xi)=f(\xi)$ for all $\xi\in X_n$, and $X_n\subset X_{2n}$. Then
\begin{equation}\label{th5.1}
  \|f_{\g/n,\,r}-f\|_{\ell_p(X_n)}+\w_s(f,1/n)_p\le C\sum_{k=1}^\infty (n2^k)^{-s}\|(G_{2^kn}(f))^{(s)}\|_p,
\end{equation}
where the constant $C$ does not depend on $f$ and $n$.
\end{theorem}

\begin{proof}
  We follow the proof of~\cite[Lemma~8]{HL}, see also~\cite{KT20}. Applying~\eqref{th1.2'} and using the condition $X_n\subset X_{2n}$, we get
  \begin{equation}\label{th5.2}
  \begin{split}
         \|f_{\g/n,\,r}-f\|_{\ell_p(X_n)}&\le 2^{1/p}\|f_{\g/n,\,r}-f\|_{\ell_p(X_{2n})}\\
    &\le \frac{2^{1/p}}{K_2}\(\|f-G_{2n}(f)\|_p+C_2\w_s(f,n^{-1})_p\).
  \end{split}
  \end{equation}
Using properties $(b_\w)$, $(c_\w)$, and $(e_\w)$, we have
  \begin{equation}\label{th5.3}
  \begin{split}
\w_s(f,n^{-1})_p
&\le 2^s\|f-G_{2n}(f)\|_p+n^{-s}\|(G_{2n}(f))^{(s)}\|_p.
  \end{split}
  \end{equation}
Thus, combining~\eqref{th5.2} and~\eqref{th5.3}, we obtain
\begin{equation}\label{th5.4}
\begin{split}
    \|f_{\g/n,\,r}-&f\|_{\ell_p(X_n)}+\w_s(f,n^{-1})_p\\
&\le \Big(2^s+\frac{2^{1/p}}{K_2}(1+2^sC_2)\Big)\|f-G_{2n}(f)\|_p+\Big(1+\frac{2^{1/p}C_2}{K_2}\Big)n^{-s}\|(G_{2n}(f))^{(s)}\|_p.
\end{split}
\end{equation}
Next, we denote
$$
I_n=\|G_{2n}-G_n(G_{2n})\|_p\quad\text{and}\quad G_k=G_k(f).
$$
By condition~\eqref{c2'}, we have
\begin{equation}\label{th5.5}
  I_n\le K_3n^{-s}\|G_{2n}^{(s)}\|_p.
\end{equation}
At the same time, taking into account that $X_n\subset X_{2n}$ and $G_n(f)(\xi)=f(\xi)$, $\xi\in X_n$, we get
\begin{equation*}
  \begin{split}
     I_n=\|G_{2n}-G_n\|_p\ge \|f-G_{n}\|_p-\|f-G_{2n}\|_p.
  \end{split}
\end{equation*}
Thus, in view of the convergence of $(G_k)_{k\in \N}$ and~\eqref{th5.5}, we obtain
\begin{equation}\label{th5.6}
  \begin{split}
     \|f-G_{n}\|_p&=\sum_{k=0}^\infty \(\|f-G_{2^kn}\|_p-\|f-G_{2^{k+1}n}\|_p\)\\
     &\le \sum_{k=0}^\infty I_{2^k n}\le 2^sK_3\sum_{k=1}^\infty (2^{k}n)^{-s}\|G_{2^kn}^{(s)}\|_p.
  \end{split}
\end{equation}
Finally, combining~\eqref{th5.4} and~\eqref{th5.6}, we arrive at~\eqref{th5.1}.
\end{proof}

Using Theorems~\ref{th5}, \ref{th4}, and~\ref{th3}, we obtain the following corollary.

\begin{corollary}\label{cor3}
  Let $f\in L_p(\T)$, $1\le p<\infty$, $r,s\in \N$, $s\le 2r$, and $\a\in (0,s)$. Suppose that $G_n(f)$, ${n\in \N}$, satisfy the conditions of Theorem~\ref{th5}. Then the following properties are equivalent:
  \begin{enumerate}
    \item[$(i)$]  $\|f-G_n(f)\|_p=\mathcal{O}(n^{-\a})$,
    \item[$(ii)$]  $\|f_{\g/n,\,r}-f\|_{\ell_p(X_n)}+\w_s(f,1/n)_p=\mathcal{O}(n^{-\a})$,
    \item[$(iii)$] $\|(G_n(f))^{(s)}\|_p=\mathcal{O}(n^{s-\a})$.
  \end{enumerate}
\end{corollary}


\subsection{Additional results and remarks}



\begin{remark}\label{remth1+}
Let $0<\d\le \min_{k}(x_{k+1,n}-x_{k,n})$. It follows from the proof of Theorem~\ref{th1+} that inequalities~\eqref{th1.1'} and~\eqref{th1.2'} can also be written in the following sharper forms:
\begin{equation*}
     \|f-G_n(f)\|_p\le K_1\|f_{\d,r}-f\|_{\ell_p(X_n)}+\frac{C_1'}{(\d n)^{1/p}}\w_s(f,1/n)_p,
\end{equation*}
where $C_1'$ is a positive constant depending on $r$, $s$, $p$, $K_1$, and $K_3$;
  \begin{equation*}
     K_2\|f_{\d,r}-f\|_{\ell_p(X_n)}-\frac{C_2'}{(\d n)^{1/p}}\w_s(f,1/n)_p\le \|f-G_n(f)\|_p,
  \end{equation*}
where $C_2'$ is a positive constant depending on $r$, $s$, $K_1$, and $K_3$.
\end{remark}

In a similar way, one can reformulate the corresponding estimates in other results of this paper.
Moreover, slightly modifying the proof of Theorem~\ref{th1+}, we obtain the following analogue of this theorem without the restriction $\d\le \min_{k}(x_{k+1,n}-x_{k,n})$. For simplicity, we consider only the case $r=s=1$.

\begin{proposition}\label{th1}
  Let $f\in L_p(\T)$, $1\le p<\infty$, $\d>0$, and $n\in \N$. Suppose that condition~\eqref{c2'} with $s=1$ holds.

  $(i)$ If additionally \eqref{c1}  holds, then
  \begin{equation}\label{th1.1}
     \|f-G_n(f)\|_p\le K_1\|f_\d-f\|_{\ell_p(X_n)}+\(1+\frac{K_3}{\d n}\)\w(f,\d)_p.
  \end{equation}

  $(ii)$  If additionally \eqref{c1'}  holds, then
  \begin{equation}\label{th1.2}
     K_2\|f_\d-f\|_{\ell_p(X_n)}-\(1+\frac{K_3}{\d n}\)\w(f,\d)_p\le \|f-G_n(f)\|_p.
  \end{equation}
\end{proposition}

\begin{proof}
$(i)$ Applying~\eqref{EQ+}, \eqref{c2'}, and \eqref{c1}, we obtain
\begin{equation*}
\begin{split}
    \|f-G_n(f)\|_p&\le \|f-f_\d\|_p+\|f_\d-G_n(f_\d)\|_p+\|G_n(f_\d-f)\|_p\\
       &\le \w(f,\d)_p+K_3n^{-1}\|f_\d'\|_p+K_1\|f_\d-f\|_{\ell_p(X_n)}\\
       &\le \w(f,\d)_p+K_3(\d n)^{-1}\w(f,\d)_p+K_1\|f_\d-f\|_{\ell_p(X_n)}.
\end{split}
\end{equation*}
This proves inequality~\eqref{th1.1}.

$(ii)$ As above, from~\eqref{c1'}, \eqref{c2'}, and \eqref{st1}, we derive
\begin{equation*}
\begin{split}
    K_2\|f_\d-f\|_{\ell_p(X_n)}&\le \|G_n(f_\d-f)\|_p\\
  &\le \|f-f_\d\|_p+\|f_\d-G_n(f_\d)\|_p+\|f-G_n(f)\|_p\\
  &\le \w(f,\d)_p+K_3n^{-1}\|f_\d'\|_p+\|f-G_n(f)\|_p\\
  &\le\(1+K_3(\d n)^{-1}\)\w(f,\d)_p+\|f-G_n(f)\|_p,
\end{split}
\end{equation*}
which implies~\eqref{th1.2}.
\end{proof}

Note that for particular classes of operators $G_n$, inequality~\eqref{th1.2'} as well as Corollary~\ref{cor1} can be specified by swapping $\|f_{1/n,\,r}-f\|_{\ell_p(X_n)}$ with $\w_s(f,1/n)_p$. Let us demonstrate this by using the trigonometric Lagrange interpolation polynomials $L_n(f)$ defined in~\eqref{Lag}.
For our purposes, we will use the so-called sampling Kantorovich operators given by
\begin{equation*}
  K_n(f)(x)=\sum_{k=0}^{2n} \int_{-\d_n/2}^{\d_n/2}f(t_k+t)dt\, D_n(x-t_k),\quad \d_n=\frac1{2n+1}.
\end{equation*}

\begin{lemma}\label{le1}
  Let $f\in L_p(\T)$, $1<p<\infty$, and $n\in \N$. Then
  \begin{equation}\label{le1.1K}
    \kappa_2 \w_2(f,1/n)_p\le \|f-K_n(f)\|_p\le \kappa_1 \w_2(f,1/n)_p,
  \end{equation}
 where $\kappa_1$ and $\kappa_2$ are some positive constants depending only on $p$.
\end{lemma}

\begin{proof}
  Inequalities~\eqref{le1.1K} can be established by repeating the proof of Theorems 4.3 and 4.4 in~\cite{KP21}, see also~\cite[Example~4.3 and Remark~4.2]{KP21}, where similar Kantorovich operators with the set of knots $(\frac k{2^j})_{k=0}^{2^j-1}$ were considered.
Note also that a non-periodic analogue of~\eqref{le1.1K} can be found in~\cite{KS21}.
\end{proof}

We need the following classical Marcinkiewicz-Zygmund inequality (see, e.g.,~\cite[p.~28]{Z}).

\begin{lemma}\label{leMZ0}
  Let $1<p<\infty$, $n\in \N$, and $X_{2n+1}^L=(t_k)_{k=0}^{2n}$, $t_k=\frac{k}{2n+1}$. Then, for each $T_n\in \mathcal{T}_n$,
  \begin{equation*}
    \mu_2 \|T_n\|_{\ell_p(X_{2n+1}^L)}\le \|T_n\|_p\le \mu_1 \|T_n\|_{\ell_p(X_{2n+1}^L)},
  \end{equation*}
 where $\mu_1$ and $\mu_2$ are some positive constants depending only on $p$.
\end{lemma}

\begin{proposition}\label{th6}
  Let $f\in L_p(\T)$, $1<p<\infty$, and $n\in \N$. Then
  \begin{equation*}\label{th6.1}
    \|f-L_n(f)\|_p\le \mu_1\|f_{1/(2n+1)}-f\|_{\ell_p(X_{2n+1}^L)}+\kappa_1\w_2(f,1/n)_p
  \end{equation*}
  and
    \begin{equation}\label{th6.2}
    \max(\Omega_1(f,n),\Omega_2(f,n))\le\|f-L_n(f)\|_p,
  \end{equation}
  where
  $$
  \Omega_1(f,n)=\kappa_2\w_2(f,1/n)_p-\mu_1\|f_{1/(2n+1)}-f\|_{\ell_p(X_{2n+1}^L)},
  $$
  $$
  \Omega_2(f,n)=\mu_2\|f_{1/(2n+1)}-f\|_{\ell_p(X_{2n+1}^L)}-\kappa_1\w_2(f,1/n)_p,
  $$
and the constants $\mu_1$, $\mu_2$ and $\kappa_1$, $\kappa_2$ are defined in Lemma~\ref{leMZ0} and Lemma~\ref{le1}, respectively.
\end{proposition}

\begin{proof}
We only prove inequality~\eqref{th6.2}. Applying Lemmas~\ref{le1} and~\ref{leMZ0}, 
we obtain
\begin{equation*}
  \begin{split}
\kappa_2\w_2(f,1/n)_p&\le \|f-K_n(f)\|_p\le \|f-L_n(f)\|_p+\|K_n(f)-L_n(f)\|_p\\
&\le \|f-L_n(f)\|_p+\mu_1\|K_n(f)-L_n(f)\|_{\ell_p(X_{2n+1}^L)}\\
&=\|f-L_n(f)\|_p+\mu_1\|f_{1/(2n+1)}-f\|_{\ell_p(X_{2n+1}^L)},
   \end{split}
\end{equation*}
which implies that $\Omega_1(f,n)\le\|f-L_n(f)\|_p$. The estimate with $\Omega_2$ (which is actually already contained in Theorem~\ref{th1+}) can be proved in a similar way.
\end{proof}

Proposition~\ref{th6} implies the following result, which supplements Corollary~\ref{cor1}.
\begin{corollary}\label{cor4}
  Let $f\in L_p(\T)$ with $1<p<\infty$.
\begin{enumerate}
  \item[$(i)$] If
  $
  \w_2(f,1/n)_p=o\(\|f_{1/(2n+1)}-f\|_{\ell_p(X_{2n+1}^L)}\),
  $
  then
  $$
  \|f-L_n(f)\|_p\sim \|f_{1/(2n+1)}-f\|_{\ell_p(X_{2n+1}^L)}.
  $$
  \item[$(ii)$] If
  $
  \|f_{1/(2n+1)}-f\|_{\ell_p(X_{2n+1}^L)}=o\(\w_2(f,1/n)_p\),
  $
  then
  $$
  \|f-L_n(f)\|_p\sim \w_2(f,1/n)_p.
  $$
\end{enumerate}
\end{corollary}

\section{Examples and comparison of measures of smoothness}

In this section, we compare the behaviour of $\w_s(f,1/n)_p$, $\tau_s(f,1/n)$, and $\|f_{1/n,\,r}-f\|_{\ell_p(X_n)}$ for some special functions and the set of nodes $X_n=(\frac kn)_{k=0}^{n-1}$. In view of property $(g_\tau)$ and Proposition~\ref{cor1+}, we are  interested in functions $f\in L_p(\T)$ that satisfy 
\begin{equation*}
  \sum_{\nu=1}^\infty \nu^{1/p-1}\w_r(f,1/\nu)_p=\infty
\end{equation*}
for some $r\in \N$. 
Otherwise, the picture is more or less clear, cf. equivalences $(i_1)-(i_5)$.

First, we note that $\w_s(f,\d)_p\le \tau_s(f,\d)_p$ for all $\d>0$ by property $(f_\tau)$. In Proposition~\ref{cor1+}, we proved that $\|f_{\d,r}-f\|_{\ell_p(X_n)}\le {(\d n)^{-1/p}}\tau_{2r}(f,\d)_p$.
Below, we show that $\w_s(f,\d)_p$ and $\|f_{\d,r}-f\|_{\ell_p(X_n)}$ are not comparable in general case.
Moreover, for some classes of functions, we show that the quantity $\|f_{\d,r}-f\|_{\ell_p(X_n)}+\w_s(f,\d)_p$ provides sharper $L_p$-error estimates of approximation by the Lagrange interpolation polynomials $L_n(f)$ than the error of best one-sided approximation $\ww E_n(f)_p$ and the modulus $\tau_s(f,\d)_p$.

In what follows, we use the notation
$\, A \lesssim B,$ with $A,B\ge 0$, for the estimate
$\, A \le C\, B,$ where $\, C$ is a positive constant independent of
the essential variables in $\, A$ and $\, B$ (usually, $f$, $j$, and $n$).
If $\, A \lesssim B$ and $\, B \lesssim A$ simultaneously, we write $\, A \asymp B$.



\begin{example}\label{ex1}
  Consider  the Dirichlet function 
\begin{equation*}
  f(x)=\left\{
         \begin{array}{ll}
           1, & \hbox{$x\in \mathbb{Q}$,} \\
           0, & \hbox{otherwise}.
         \end{array}
       \right.
\end{equation*}
It is not difficult to see that
$$
\|f_{1/n}-f\|_{\ell_p(X_{n})}=\tau(f,1/n)_p=1
$$
while
$$
\w(f,1/n)_p=0.
$$
Moreover,
$$
\|f-L_n(f)\|_p=\|f_{1/(2n+1)}-f\|_{\ell_p(X_{2n+1})}+\w(f,1/n)_p=\tau(f,1/n)_p=1.
$$
\end{example}


For the following modification of the Dirichlet function, the averaged modulus of smoothness $\tau(f,1/n)_p$ is essentially bigger than $\|f-L_n(f)\|_p$.  

\begin{example}\label{ex1+}
  Let 
\begin{equation*}
  f(x)=\left\{
         \begin{array}{ll}
           1, & \hbox{$x=\frac ab\in \mathbb{Q}$, $(a,b)=1$ and $b$ is even,} \\
           0, & \hbox{otherwise}.
         \end{array}
       \right.
\end{equation*}
Then
$$
\|f_{1/(2n+1)}-f\|_{\ell_p(X_{2n+1})}=\w(f,1/n)_p=0
$$
while
$$
\tau(f,1/n)_p=1.
$$
Moreover,
$$
\|f-L_n(f)\|_p=\|f_{1/(2n+1)}-f\|_{\ell_p(X_{2n+1})}+\w(f,1/n)_p=0.
$$
\end{example}

%


In the next example, we show that $\|f_{1/n}-f\|_{\ell_p(X_{n})}$ can be essentially smaller than $\w(f,\d)_p$ and $\tau(f,\d)_p$.

\begin{example}
Let $1\le p<\infty$, $r\in \N$, and
$$
f(x)=\left\{
       \begin{array}{ll}
         0, & \hbox{$x\in\{0,1/2\}$,} \\
         1, & \hbox{$0<x<1/2$,} \\
         -1, & \hbox{$1/2<x<1$,}
       \end{array}
     \right.
$$
with $f(x+1)=f(x)$ for $x\in\R\setminus \T$.
Then
$$
\|f_{1/(2n)}-f\|_{\ell_p(X_{2n})}=0
$$
but
$$
\w_r(f,n^{-1})_p\asymp\tau_r(f,n^{-1})_p\asymp n^{-\frac1p}.
$$
\end{example}

We now show that the quantity  $\|f_{1/n}-f\|_{\ell_p(X_n)}$ can be used for the investigation of unbounded functions in $L_p(\T)$.
Note that the best one-sided approximation and in particular the averaged moduli of smoothness make sense only for bounded functions.


\begin{example}\label{ex2}
  Let $1\le p<\infty$, $0<\a<1/p$, $r\in \N$, and let
\begin{equation*}
  f(x)=\left\{
         \begin{array}{ll}
           0, & \hbox{$x=0$,} \\
           \frac1{x^\a}, & \hbox{$0<x<1$,}
         \end{array}
       \right.
\end{equation*}
with $f(x)=f(x+1)$ for $x\in \R\setminus\T$.
Then
\begin{equation*}
  \|f_{1/n}-f\|_{\ell_p(X_n)}\asymp \w_r(f,1/n)_p\asymp n^{-(1/p-\a)}.
\end{equation*}
Moreover, for $1<p<\infty$, we have
\begin{equation}\label{ex3.3++}
\|f-L_{n}(f)\|_p\asymp \|f_{1/(2n+1)}-f\|_{\ell_p(X_{2n+1})}+\w_r(f,1/n)_p \asymp E_n(f)_p\asymp n^{-(1/p-\a)}.
\end{equation}
\end{example}

\begin{proof}
Using the standard calculations and properties of moduli of smoothness, one can verify that
\begin{equation}\label{ex3.2}
\w_r(f,\d)_p\asymp \d^{1/p-\a}
\end{equation}
for sufficiently small $\d>0$. Simple calculations also yield
\begin{equation}\label{ex3.3}
  \|f_{1/n}-f\|_{\ell_p(X_n)}\asymp n^{-(1/p-\a)}.
\end{equation}
Indeed,
\begin{equation}\label{ex3.4}
  \|f_{1/n}-f\|_{\ell_p(X_n)}\ge n^{-1/p}|f_{1/n}(0)|=\frac1{(1-\a)2^{(1-\a)}}n^{-(1/p-\a)}+\mathcal{O}(n^{-1/p}).
\end{equation}
To obtain the upper estimate, we note that for $k\in [1,n-1]$, by the mean value theorem
\begin{equation*}
\begin{split}
   |f_{1/n}\(k/n\)-f\(k/n\)|\lesssim \frac{n^\a}{k^{\a+1}}.
\end{split}
\end{equation*}
Hence, $\|f_{1/n}-f\|_{\ell_p(X_n)}\lesssim n^{-(1/p-\a)}$,
which together with~\eqref{ex3.4} yields~\eqref{ex3.3}

Now, by Proposition~\ref{th6} and~\eqref{ex3.2}, \eqref{ex3.3}, we obtain
\begin{equation}\label{ex3.5}
\|f-L_n(f)\|_p\lesssim \|f_{1/(2n+1)}-f\|_{\ell_p(X_{2n+1})}+\w_2(f,1/(2n+1))_p\lesssim n^{-(1/p-\a)}.
\end{equation}
At the same time, by Lemma~\ref{lemR} and~\eqref{ex3.2}, we have
$$
\|f-L_n(f)\|_p\ge E_n(f)_p \asymp n^{-(1/p-\a)}.
$$
These relations and~\eqref{ex3.5} imply~\eqref{ex3.3++}.
%
\end{proof}

\bigskip

In the next two examples, we consider continuous oscillating functions. To simplify calculations, we use the modified Lagrange interpolation polynomials $L_n^*(f)$ given by
$$
L_n^*(f)(x)=2^{-n}\sum_{k=0}^{2^n-1}f(k2^{-n})D_{2^n}^*(x-k2^{-n}),
$$
where
$$
D_{2^n}^*(x)=\sum_{\nu=-2^{n-1}}^{2^{n-1}-1}e^{2\pi{\rm i}\nu x}.
$$

\begin{example}\label{pr4}
  Let $1\le p<\infty$, $0<\b<1/p$, $r\in \N$, and let
\begin{equation*}
  f(x)=\sum_{\ell=3}^\infty \ell^{-\b}(1-{4^{\ell}}|x-2^{-\ell}|)_+,\quad x\in \T,
\end{equation*}
with $f(x)=f(x+1)$ for $x\in \R\setminus\T$.
Then, for sufficiently large $n$,
\begin{equation}\label{pr4.1}
  \|f_{1/2^n}-f\|_{\ell_p(X_{2^n})}\asymp 2^{-n/p}n^{1/p-\b},
\end{equation}

\begin{equation}\label{pr4.2}
\tau_r(f,2^{-n})_p\asymp 2^{-n/p}n^{1/p-\b},
\end{equation}

\begin{equation}\label{pr4.3}
\w_r(f,2^{-n})_p\asymp 2^{-n/p}n^{-\b}.
\end{equation}
Moreover, for $1<p<\infty$, we have
\begin{equation}\label{pr4.4}
\|f-L_{2^n}^{*}(f)\|_p\asymp \ww E_{2^n}(f)_p \asymp 2^{-n/p}n^{1/p-\b}\quad\text{and}\quad E_{2^n}(f)_p \asymp 2^{-n/p}n^{-\b}.
\end{equation}

\end{example}

\begin{proof}
First, we estimate $\|f_{1/2^n}-f\|_{\ell_p(X_{2^n})}$. For simplicity, we assume that $n$ is odd. Denote $\vp_\ell(x)=\ell^{-\b}(1-{4^{\ell}}|x-2^{-\ell}|)_+$ and $y_\ell=\ell^{-\b}$, $\ell\ge 3$. Direct calculations show that
$$
f(0)=0,\quad f_{1/2^n}(0)=2^n\sum_{\ell=n+2}^\infty y_\ell 4^{-\ell}+y_{n+1}2^{n-1}4^{-n-1}\asymp y_{n+1}2^{-n-1},
$$
$$
f(2^{-n})=y_n,\quad f_{1/2^n}(2^{-n})=y_{n+1}2^{n-1}4^{n+1}+y_n2^n4^{-n}\asymp y_n 2^{-n},
$$
and
$$
f(2^{-\ell})=y_{\ell},\quad f_{1/2^n}(2^{-\ell})=y_\ell 2^{-(2\ell-n)}\quad\text{for}\quad \ell=\frac{n+1}2,\dots,n-1.
$$
Here, we also took into account that $2^{-\ell-1}+2^{-n-1}>2^{-\ell-1}+4^{-\ell-1}$  and $2^{-\ell}-2^{-n-1}\le 2^{-\ell}-4^{-\ell}$  for $\ell=\frac{n+1}2,\dots,n-1$.

Denote
$$
F=\sum_{\ell=\frac{n+1}{2}}^\infty \vp_\ell\quad\text{and}\quad\quad H=f-F.
$$
It is not difficult to see that
\begin{equation}\label{ex4.3}
  \|f_{1/2^n}-f\|_{\ell_p(X_{2^n})}=\|F_{1/2^n}-F\|_{\ell_p(X_{2^n})}+\|H_{1/2^n}-H\|_{\ell_p(X_{2^n})}
\end{equation}
and
\begin{equation}\label{ex4.4}
  \sum_{\nu=0}^{2^n-1}|F_{1/2^n}(\nu2^{-n})-F(\nu2^{-n})|^p\asymp y_{n+1}^p2^{-np}+\sum_{\ell=\frac{n+1}{2}}^ny_\ell^p(1-2^{-(2\ell-n)})^p \asymp n^{1-\b p},
\end{equation}
which yields 
\begin{equation}\label{ex4.5}
  \|F_{1/2^n}-F\|_{\ell_p(X_{2^n})}\asymp 2^{-n/p}n^{1/p-\b}.
\end{equation}
To estimate $\|H_{1/2^n}-H\|_{\ell_p(X_{2^n})}$, we introduce  the sets
$$
A_{\ell,n}=\{\nu\in [0,2^n-1]\,:\,(\nu2^{-n}-2^{-n-1},\nu2^{-n}+2^{-n-1})\cap (2^{-\ell}-4^{-\ell},2^{-\ell}+4^{-\ell})\neq \varnothing\}.
$$
We have
\begin{equation}\label{ex4.6}
  \begin{split}
     \sum_{\nu=0}^{2^n-1}|G_{1/2^n}(\nu2^{-n})-G(\nu2^{-n})|^p=\sum_{\ell=3}^{\frac{n-1}{2}}\sum_{\nu\in A_{\ell,n}}|(\vp_\ell)_{1/2^n}(\nu2^{-n})-\vp_\ell(\nu2^{-n})|^p.
  \end{split}
\end{equation}
Then, taking into account that $h_{\d}(x)=h(x)$ for any linear function $h$  and that
\begin{equation*}
(\vp_\ell)_{1/2^n}(2^{-\ell}\pm 4^{-\ell})=y_\ell 4^\ell 2^{-n-3},
\end{equation*}
\begin{equation*}
(\vp_\ell)_{1/2^n}(2^{-\ell})=y_\ell-y_\ell 4^\ell 2^{-n-2},
\end{equation*}
we obtain
\begin{equation*}
\begin{split}
   \sum_{\nu\in A_{\ell,n}}&|(\vp_\ell)_{1/2^n}(\nu2^{-n})-\vp_\ell(\nu2^{-n})|^p=|(\vp_\ell)_{1/2^n}(2^{-\ell}-4^{-\ell})|^p\\
&+|(\vp_\ell)_{1/2^n}(2^{-\ell})-(\vp_\ell)(2^{-\ell})|^p+|(\vp_\ell)_{1/2^n}(2^{-\ell}+4^{-\ell})|^p\\
&=2(y_\ell 4^\ell 2^{-n-3})^p+(y_\ell 4^\ell 2^{-n-2})^p,
\end{split}
\end{equation*}
which together with~\eqref{ex4.6} implies 
\begin{equation}\label{ex4.8}
  \|G_{1/2^n}-G\|_{\ell_p(X_{2^n})}\asymp \(2^{-n}\sum_{\ell=3}^{\frac{n-1}{2}}(y_\ell 4^\ell 2^{-n})^p\)^{1/p}\asymp 2^{-n/p}n^{-\b}.
\end{equation}
Thus, combining~\eqref{ex4.3}, \eqref{ex4.5}, and~\eqref{ex4.8}, we get~\eqref{pr4.1}.

Let us prove~\eqref{pr4.3}. By $(c_\w)$ and $(e_\w)$, we obtain
\begin{equation}\label{ex4.12}
\begin{split}
\w(f,2^{-n})_p&\le \w(f,2^{-n})_{L_p[0,2^{-\frac{n+1}2}]}+\w(f,2^{-n})_{L_p[2^{-\frac{n+1}2},1]}\\
&\lesssim \|f\|_{L_p[0,2^{-\frac{n+1}2}]}+2^{-n}\|f'\|_{L_p[2^{-\frac{n+1}2},1]}\\
&\lesssim \sum_{\ell=\frac{n+1}2}^\infty (y_\ell)^p 4^{-\ell}+2^{-n}\bigg(\sum_{\ell=3}^{\frac{n-1}{2}} (y_\ell 4^\ell)^p4^{-\ell}\bigg)^{1/p}\asymp 2^{-n/p}n^{-\b}.
\end{split}
\end{equation}
We now estimate $\w_2(f,2^{-n})_p$ from below for $n>5$. We get
\begin{equation}\label{ex4.13}
\begin{split}
\w_2(f,2^{-n})_p^p&\ge \|\D_{2^{-n-2}}^2 f\|_p^p\\
&\ge \sum_{\ell=\frac{n+3}{2}}^{n-1}\int_{2^{-\ell}-4^{-\ell}}^{2^{-\ell}+4^{-\ell}}|f(x+2^{-n-1}) -2f(x+2^{-n-2})+ f(x)|^pdx\\&=\sum_{\ell=\frac{n+3}{2}}^{n-1}\int_{2^{-\ell}-4^{-\ell}}^{2^{-\ell}+4^{-\ell}}|f(x)|^pdx,
\end{split}
\end{equation}
where the above equality follows from the fact that $f(x)=0$ for $x\in [2^{-\ell-1}+4^{-\ell-1},2^{-\ell}-4^{-\ell}]\cup [2^{-\ell}+4^{-\ell},2^{-\ell+1}-4^{-\ell+1}]$. Indeed, if $\ell=\frac{n+3}2,\dots,n-1$, then
$[2^{-\ell-1}+4^{-\ell-1}-2^{-n-1},2^{-\ell}-4^{-\ell}-2^{-n-2}]\subset [2^{-\ell-1}+4^{-\ell-1},2^{-\ell}-4^{-\ell}]$, which together with the inequality $2^{-\ell}+4^{-\ell}\le 2^{-\ell+1}-4^{\ell+1}-2^{-n-1}$ implies that $f(x+2^{-n-2})=f(x+2^{-n-1})=0$ for $x\in [2^{-\ell}-4^{-\ell},2^{-\ell}+4^{-\ell}]$.
Next, we have
\begin{equation*}
\begin{split}
\int_{2^{-\ell}-4^{-\ell}}^{2^{-\ell}+4^{-\ell}}|f(x)|^pdx\ge y_\ell^p\int_{2^{-\ell}-4^{-\ell}}^{2^{-\ell}}|4^\ell(x-2^{-\ell})+1|^pdx=
\frac{y_\ell^p}{(p+1)4^{\ell}},
\end{split}
\end{equation*}
which together with~\eqref{ex4.13} and~\eqref{ex4.12} yields
\begin{equation}\label{zvez}
  \w_2(f,2^{-n})_p\asymp \w(f,2^{-n})_p\asymp 2^{-n/p}n^{-\b}.
\end{equation}
Thus, by Lemma~\ref{lemR} and~\eqref{zvez}, we get~\eqref{pr4.3}. 

Now, we prove~\eqref{pr4.2}. We have
\begin{equation}\label{ex4.10}
\begin{split}
    \tau_2(f,2^{-n})_p^p&\ge \sum_{\ell=\frac{n+3}{2}}^{n-1} \int_{2^{-\ell}}^{2^{-\ell}+2^{-n-1}}\w_2(f,x,2^{-n})^pdx\\
&\ge \sum_{\ell=\frac{n+3}{2}}^{n-1} \int_{2^{-\ell}}^{2^{-\ell}+2^{-n-1}}|\D_{4^{-\ell}}^2f(2^{-\ell})|^p dx\\
&\ge 2^{-n-1}\sum_{\ell=\frac{n+3}{2}}^{n-1} y_\ell^p\asymp 2^{-n}n^{1-\b p}.
\end{split}
\end{equation}
In the above formula, we use the fact that $f(2^{-\ell}+4^{-\ell})=f(2^{-\ell}+2\cdot4^{-\ell})=0$ and $2^{-\ell},2^{-\ell}+2\cdot 4^{-\ell}\in [x-2^{-n-1},x+2^{-n-1}]$ for all $x\in [2^{-\ell},2^{-\ell}+2^{-n-1}]$ and $\ell=\frac{n+3}2,\dots,n-1$.
Next,
\begin{equation}\label{ex4.11}
\begin{split}
    \tau(f,2^{-n})_p^p&\lesssim \int_0^{2^{-n}} \w(f,x,2^{-n-1})^pdx\\
&+\bigg(\sum_{\ell=\frac{n+1}{2}}^{n-1}\,\,+\sum_{\ell=3}^{\frac{n-1}{2}}\,\,\bigg)
\int_{2^{-\ell}-4^{-\ell}-2^{-n-1}}^{2^{-\ell}+4^{-\ell}+2^{-n-1}}\w(f,x,2^{-n-1})^pdx\\
&\lesssim 2^{-n}y_n^p+2^{-n}\sum_{\ell=\frac{n+1}{2}}^{n-1}y_\ell^p+\sum_{\ell=3}^{\frac{n-1}{2}}(4^\ell y_\ell 2^{-n})^p 4^{-\ell}\asymp 2^{-n}n^{1-\b p}.
\end{split}
\end{equation}
Thus, inequalities~\eqref{ex4.10} and~\eqref{ex4.11} imply that
$$
\tau(f,2^{-n})_p\asymp \tau_2(f,2^{-n})_p\asymp 2^{-n}n^{1-\b p}.
$$
These equivalences and Lemma~\ref{lemR+} yield~\eqref{pr4.2}.

Finally, relations~\eqref{pr4.4} follow from Theorem~\ref{th1+}, relations~\eqref{pr4.1}--\eqref{pr4.3} and Lemmas~\ref{lemR} and~\ref{lemR+}.
\end{proof}

Slightly modifying the function in Example~\ref{pr4}, we can construct a continuous function $f$ such that the quantity $\|f_{1/2^n}-f\|_{\ell_p(X_{2^n})}+\w_r(f,2^{-n})_p$ provides sharp estimate for $\|f-L_{2^n}^{*}(f)\|_p$ while $\tau_r(f,2^{-n})_p$ and $\ww E_{2^n}(f)_p$ do not. 

\begin{example}\label{pr5}
  Let $1\le p<\infty$, $0<\b<1/p$, $r\in \N$, and
\begin{equation*}
  f(x)=\sum_{\ell=4}^\infty \ell^{-\b}(1-{4^{\ell}}|x-(2^{-\ell}+4^{-\ell})|)_+,\quad x\in \T,
\end{equation*}
with $f(x)=f(x+1)$ for $x\in \R\setminus\T$.
Then, for sufficiently large $n$,
\begin{equation}\label{pr5.1}
  \|f_{1/2^n}-f\|_{\ell_p(X_{2^n})}\asymp 2^{-n/p}n^{-\b},
\end{equation}

\begin{equation*}
\w_r(f,2^{-n})_p\asymp 2^{-n/p}n^{-\b},
\end{equation*}

\begin{equation*}
\tau_r(f,2^{-n})_p\asymp 2^{-n/p}n^{1/p-\b}.
\end{equation*}
Moreover, for $1<p<\infty$, we have
\begin{equation*}
\|f-L_{2^n}^{*}(f)\|_p\asymp E_{2^n}(f)_p\asymp  2^{-n/p}n^{-\b}\quad\text{but}\quad \ww E_{2^n}(f)_p \asymp 2^{-n/p}n^{1/p-\b}.
\end{equation*}
\end{example}

\begin{proof}
  The assertion can be established by repeating step by step the proof of Example~\ref{pr4}. Here, we only mention that unlike to the previous calculations, we have
$f(2^{-\ell})=0$ for all $l=\frac{n+1}{2},\dots,n-1$ and, therefore,
the corresponding analogue of~\eqref{ex4.4} has the form
\begin{equation*}
  \sum_{\nu=0}^{2^n-1}|F_{1/2^n}(\nu2^{-n})-F(\nu2^{-n})|^p\asymp y_{n+1}^p2^{-np}+\sum_{\ell=\frac{n+1}{2}}^ny_\ell^p(2^{-(2\ell-n)})^p \asymp n^{-\b p},
\end{equation*}
which results in~\eqref{pr5.1}.
\end{proof}

\bigskip




\begin{thebibliography}{16}
{


\bibitem{BXZ92} P. B. Borwein,   T. F. Xie,  S. P. Zhou,
On approximation by trigonometric Lagrange interpolating polynomials II,
Bull. Austral. Math. Soc. \textbf{45} (1992), no.~2, 215--221.

\bibitem{CZ93_2} C. K. Chui, X. C. Shen, L. Zhong, On Lagrange polynomial quasi-interpolation. Topics in polynomials of one and several variables and their applications, 125--141, World Sci. Publ., River Edge, NJ, 1993.

\bibitem{CZ93} C. K. Chui, X. C. Shen, L. Zhong,  On Lagrange interpolation at disturbed roots of unity, Trans. Amer. Math. Soc. \textbf{336} (1993), no.~2, 817--830.

\bibitem{CZ99} C. Chui, L. Zhong, Polynomial interpolation and Marcinkiewicz-Zygmund inequalities on the unit circle,
J. Math. Anal. Appl. \textbf{233} (1999), no.~1, 387--405.

\bibitem{DL} {R. A. DeVore, G. G. Lorentz},  Constructive Approximation, Springer-Verlag, New York, 1993.

\bibitem{DI93} Z. Ditzian, K.G. Ivanov, Strong converse inequalities, J. Anal. Math. \textbf{61} (1993), 61--111.

\bibitem{H83} V. Kh. Khristov,  Mean convergence of interpolation polynomials of periodic functions. (Russian) Pliska Stud. Math. Bulgar. \textbf{5} (1983), 14--22.


\bibitem{H} V. H. Hristov, Best onesided approximation and mean approximation by interpolation polynomials
of periodic functions, Math. Balkanica, New Series \textbf{3}: Fasc. 3-4 (1989), 418--429.

\bibitem{HI90} V. H. Hristov, K. G. Ivanov, Realization of $K$-functionals on subsets and constrained approximation,
Math. Balkanica (New Series) \textbf{4}  (1990), no.~3,  236--257.

\bibitem{HY95} Y. Hu, X. M. Yu, Discrete modulus of smoothness of splines with equally spaced knots, SIAM J. Numer. Anal. \textbf{32} (1995), no.~5, 1428--1435.

\bibitem{HL} Y. Hu, Y. Liu, On equivalence of moduli of smoothness of polynomials in $L_p, 0 < p \le \infty$, J.
Approx. Theory \textbf{136} (2005), no. 2, 182--197.

\bibitem{K12} Yu. S. Kolomoitsev, Approximation properties of generalized Bochner-Riesz means in the Hardy spaces $H_p,\ 0<p\leq1$. (Russian) Mat. Sb. \textbf{203} (2012), no. 8, 79--96; translation in Sb. Math. \textbf{203} (2012), no.~7-8, 1151--1168.


\bibitem{K17} Yu. Kolomoitsev, On moduli of smoothness and averaged differences of fractional order, Fract. Calc. Appl. Anal. \textbf{20} (2017), no.~4, 988--1009.

\bibitem{KKS20} Yu. Kolomoitsev, A. Krivoshein, M. Skopina, Approximation by periodic multivariate quasi-projection operators, J. Math. Anal. Appl. \textbf{489} (2020), no.~2, 124192.

\bibitem{KLP}    Yu. Kolomoitsev, T. Lomako, J. Prestin, On $L_p$-error of bivariate polynomial interpolation on the square, J. Approx. Theory \textbf{229} (2018), 13--35.

\bibitem{KP21}  Yu. Kolomoitsev, J. Prestin, Approximation properties of periodic multivariate quasi-interpolation operators, J. Approx. Theory \textbf{270} (2021), 105631.

\bibitem{KS21}  Yu. Kolomoitsev, M. Skopina, Approximation by multivariate quasi-projection operators and Fourier multipliers, Appl. Math. Comput. \textbf{400} (2021), 125955.

\bibitem{KT20} Yu. S. Kolomoitsev, S. Yu. Tikhonov, Smoothness of functions versus smoothness of approximation processes, Bull. Math. Sci. \textbf{10} (2020), no.~3, 2030002.

\bibitem{KT20_2} Yu. Kolomoitsev, S. Tikhonov,  Properties of moduli of smoothness in $L_p(\R^d)$,
J. Approx. Theory \textbf{257} (2020), 105423.

\bibitem{KT21} Yu. Kolomoitsev, S. Tikhonov,  Hardy-Littlewood and Ulyanov inequalities, Mem. Amer. Math. Soc. \textbf{271}  (2021), no.~1325.

\bibitem{KT12} Yu. S. Kolomoitsev, R. M. Trigub, On the nonclassical approximation method for periodic functions by trigonometric polynomials (Russian) Ukr. Mat. Visn. \textbf{9} (2012), no. 3, 356--374; translation in J.~Math. Sci. \textbf{188} (2013), no. 2, 113--127.

\bibitem{L40} S. Lozinski, \"Uber trigonometrische Interpolation, Izv. Akad. Nauk SSSR Ser. Mat. \textbf{4} (1940), no.~2, 229--248.

\bibitem{LMN} {D. S. Lubinsky, A. Mate, P. Nevai}, Quadrature sums involving $p$th powers of polynomials, {SIAM J. Math. Anal.} \textbf{18} (1987), 53--544.


\bibitem{MS09} J. Marzo, K. Seip, The Kadets 1/4 theorem for polynomials, Math. Scand.
\textbf{104} (2009), no.~2, 311--318.


%




\bibitem{O86} K. I. Oskolkov, Inequalities of the "large sieve'' type and applications to problems of trigonometric approximation, Anal. Math. \textbf{12} (1986), no.~2, 143--166.

%
%
%
%
%
%
%

\bibitem{Pr84} J. Prestin, Trigonometric interpolation of functions of bounded variation in: Constructive Theory of Functions (Proc. of the Intern. Conf. on Constr. Theory of Functions, Varna 1984, Ed. B. Sendov) Sofia 1984, 699--703.


\bibitem{PQ92} J. Prestin, E. Quak,
On interpolation and best one-sided approximation by splines in $L_p$, Approximation theory (Memphis, TN, 1991), 409--420,
Lecture Notes in Pure and Appl. Math. \textbf{138}, Dekker, New York, 1992.


\bibitem{PrXu} J. Prestin, Y. Xu, Convergence rate for trigonometric interpolation of non-smooth functions, J. Approx. Theory \textbf{77} (1994), no. 2, 113--122.
%

\bibitem{P83} V. A. Popov, Function spaces, generated by the averaged moduli of smoothness, PLISKA, Stud. Math. Bulg \textbf{5} (1983), 132--143.

\bibitem{P84} V. A. Popov, A one-sided $K$-functional and its interpolation spaces. (Russian) International conference on analytical methods in number theory and analysis (Moscow, 1981). Trudy Mat. Inst. Steklov. \textbf{163} (1984), 196--199.

\bibitem{R94} R. Rathore, The problem of A. F. Timan on the precise order of decrease of the best approximations, J.~Approx. Theory \textbf{77} (1994), no.~2, 153--166.


\bibitem{SP} B. Sendov, V. A. Popov, The Averaged Moduli of Smoothness. Chichester: John Wiley \& Sons Ltd., 1988.


\bibitem{timan} A. F. Timan, Theory of Approximation of Functions of a Real Variable, Pergamon Press, Oxford, London, New York, Paris, 1963.

%
%
%
%

\bibitem{T80} R. M. Trigub,  Absolute convergence of Fourier integrals, summability of Fourier series and approximation by polynomials of functions on a torus (Russian), Izv. Akad. Nauk SSSR Ser. Mat. \textbf{44} (1980), no.~6, 1378--1409.


\bibitem{T13} R. M. Trigub, The exact order of approximation to periodic functions by Bernstein-Stechkin polynomials, Sb. Math. \textbf{204} (2013), no.~12, 1819--1838.


\bibitem{TB}  R. M.~Trigub, E. S.~Belinsky, Fourier Analysis
and Approximation of Functions.  Kluwer-Springer (2004).

%


\bibitem{V82} P. V\'ertesi, On the almost everywhere divergence of Lagrange interpolation (complex and trigonometric cases),
Acta Math. Acad. Sci. Hungar. \textbf{39} (1982), no.~4, 367--377.


\bibitem{WS03} J. Wang, S. Zhou, Some converse results on one-sided approximation: justifications,
Anal. Theory Appl. \textbf{19} (2003), no.~3, 280--288.



\bibitem{Z} A. Zygmund, Trigonometric Series, Vol. II,  Cambridge University Press,  1968.

}

\end{thebibliography}
\end{document}